\documentclass[11pt]{article}
\usepackage[utf8]{inputenc}
\usepackage{lmodern}
\usepackage{subfiles}
\usepackage{enumitem}
	\setenumerate{label={\normalfont(\roman*)}, itemsep=0em} 
        
\usepackage{amsfonts}
\usepackage{amsthm}
\usepackage{amsmath}
\usepackage{amssymb}
\usepackage{amscd}
\usepackage{mathrsfs}
\usepackage{mathtools}
\usepackage{bm}
\usepackage{esint}

\usepackage[margin=3cm]{geometry}
\usepackage{indentfirst}
\usepackage{graphicx}
\usepackage{graphics}
\usepackage{lscape}
\usepackage{tikz-cd}
\usepackage{color}
\usepackage{pict2e}
\usepackage{epic}
\usepackage{epstopdf}
\usepackage{titlesec, titlefoot}
	\titleformat{\section}[block]{\Large\bfseries\filcenter}{\thesection}{1em}{}
\usepackage{commath}
\usepackage{float}
\usepackage{caption}
\usepackage{etoolbox}
\usepackage[affil-it]{authblk}
\usepackage{combelow}

\usepackage[hidelinks]{hyperref}
\hypersetup{bookmarksopen=true} 
\usepackage{hypcap}
\usepackage{schemata}

\graphicspath{{./Pictures/}}
\allowdisplaybreaks

\usepackage{listings}
\usepackage{fancybox}

\theoremstyle{plain}

\renewcommand*\thesection{\arabic{section}}
\numberwithin{equation}{section} 

\theoremstyle{plain}
\newtheorem{thm}{Theorem}

\newtheorem{lemma}[thm]{Lemma}

\newtheorem{cor}{Corollary}
\numberwithin{thm}{section} 

\newtheorem{theorem}{Theorem}
\newtheorem{proposition}{Proposition}

\theoremstyle{definition}

\newtheorem{remark}[thm]{Remark}

\newtheorem{definition}[thm]{Definition}

\newcommand{\thistheoremname}{}
\newtheorem{genericthm}[equation]{\thistheoremname}

\newcommand{\thistheoremnames}{}
\newtheorem*{genericthms}{\thistheoremnames}
\newenvironment{para*}[1]
  {\renewcommand{\thistheoremnames}{#1}%
   \begin{genericthms}}
  {\end{genericthms}}

\expandafter\let\expandafter\oldproof\csname\string\proof\endcsname
\let\oldendproof\endproof
\renewenvironment{proof}[1][\proofname]{%
  \oldproof[\upshape \bfseries #1]%
}{\oldendproof}

\makeatletter
\def\@makechapterhead#1{%
  \vspace*{50\p@}%
  {\parindent \z@ \raggedright \normalfont
    \interlinepenalty\@M
    \Huge\bfseries  \thechapter.\quad #1\par\nobreak
    \vskip 40\p@
  }}
\makeatother

\newcommand{\reqnomode}{\tagsleft@false}
\def \d{\,{\rm d}}

\def\supp{\,{\rm supp }}

\allowdisplaybreaks
\makeatletter
\DeclareRobustCommand*{\bfseries}{%
  \not@math@alphabet\bfseries\mathbf
  \fontseries\bfdefault\selectfont
  \boldmath
}

\makeatother

\newlength{\defbaselineskip}
\newcommand{\N}{\mathbb{N}}
\newcommand\eps\varepsilon
\def\mean#1{\mathchoice%
          {\mathop{\kern 0.2em\vrule width 0.6em height 0.69678ex depth -0.58065ex
                  \kern -0.8em \intop}\nolimits_{\kern -0.4em#1}}%
          {\mathop{\kern 0.1em\vrule width 0.5em height 0.69678ex depth -0.60387ex
                  \kern -0.6em \intop}\nolimits_{#1}}%
          {\mathop{\kern 0.1em\vrule width 0.5em height 0.69678ex
              depth -0.60387ex
                  \kern -0.6em \intop}\nolimits_{#1}}%
          {\mathop{\kern 0.1em\vrule width 0.5em height 0.69678ex depth -0.60387ex
                  \kern -0.6em \intop}\nolimits_{#1}}}

\numberwithin{equation}{section}
\allowdisplaybreaks[4]

\def\eqn#1$$#2$${\begin{equation}\label#1#2\end{equation}}
\newcommand\R{\mathbb{R}}

\newcommand{\F}{\mathscr F}
\newcommand{\G}{\mathscr G}

\def \tp{\textup}
\def \p{\partial}
\def \e{\varepsilon}
\def \D{\mathrm{D}}
\def \S{\mathbb S}

\newcommand\restr[2]{{
  \left.\kern-\nulldelimiterspace 
  #1 
  \vphantom{|} 
  \right|_{#2} 
  }}

\title{Global higher integrability for minimisers of convex obstacle problems with (\MakeLowercase{p},\MakeLowercase{q})-growth}
\author{Lukas Koch}

\affil[1]{\small University of Oxford, Andrew Wiles Building Woodstock Rd, Oxford OX2 6GG, United Kingdom 
\protect \\
  {\tt{kochl@maths.ox.ac.uk}}
  \vspace{1em} \ }

\usepackage{etoolbox}
\makeatletter
\patchcmd{\@adjustvertspacing}
  {\jot\baselineskip \divide\jot 4}
  {\jot=.5\baselineskip}
  {}{}
\makeatother

\makeatother
\begin{document}
\maketitle
\begin{abstract}
We prove global $W^{1,q}(\Omega,\R^N)$-regularity for minimisers of $\F(u)=\int_\Omega F(x,Du)\d x$ satisfying $u\geq \psi$ for a given Sobolev obstacle $\psi$. $W^{1,q}(\Omega,\R^m)$ regularity is also proven for minimisers of the associated relaxed functional. Our main assumptions on $F(x,z)$ are a uniform $\alpha$-H\"older continuity assumption in $x$ and natural $(p,q)$-growth conditions in $z$ with $q<\frac{(n+\alpha)p}{n}$. In the autonomous case $F\equiv F(z)$ we can improve the gap to $q<\frac{np}{n-1}$, a result new even in the unconstrained case.
\end{abstract}

\unmarkedfntext{
\hspace{-0.8cm}
\emph{Acknowledgments:} L.K. was supported by the Engineering and Physical Sciences Research Council [EP/L015811/1].
}

\section{Introduction and results}
We are interested in the following vectorial obstacle problem: given a domain $\Omega\subset\R^n$, boundary datum $g$ and an obstacle $\psi$, consider the obstacle problem
\begin{align}\label{obstacleProblem}
\min_{u\in K^\psi_g(\Omega)} \F(u)\quad \text{ where }\quad \F(u)=\int_\Omega F(x,\D u)\d x. \tag{P}
\end{align}
Here we write
\begin{align*}
K^\psi_g(\Omega) = \{u\in W^{1,p}_g(\Omega,\R^N)\colon u\geq \psi \text{ a.e. in } \Omega\}.
\end{align*}
Here and throughout we understand vector-valued inequalities such as $u\geq \psi$ to be applied row-wise, that is if $u=(u_1,\ldots,u_N)^T$, $\psi=(\psi_1,\ldots,\psi_N)^T$, then $u_i\geq \psi_i$ for ${i=1,\ldots,N}$.
In order to ensure that $K^\psi_g(\Omega)$ is non-empty we assume throughout that $g,\psi$ are Sobolev functions satisfying $g\geq \psi$ on $\p\Omega$ in the sense of traces. We refer to Section \ref{sec:prelim} for unexplained notation.

The integrand $F\equiv F(x,z)\colon \Omega \times \R^{N\times n}\to \R$ is convex  and satisfies $(p,q)$-growth in $z$ as well as a natural uniform $\alpha$-H\"older condition in $x$. We stress that we work in the multi-dimensional vectorial case $n>1$, $N\geq 1$. We prove global higher integrability properties of the minimiser as well as global higher integrability properties of minimisers of a relaxed functional related to \eqref{obstacleProblem}. We also study the autonomous case $F\equiv F(z)$ where we obtain stronger results which are in fact new even in the unconstrained case where no obstacle is present.

In order to state our results precisely and to compare them with the literature on (constrained) functionals with $(p,q)$-growth we list our assumptions.

We suppose throughout that $F\equiv F(x,z)\colon \Omega\times \R^{N\times n}\to \R$ is measurable in $x$ and continuously differentiable in $z$. We further introduce the following assumptions, assumed to hold for all $z\in \R^{N\times n}$ and almost every $x,y\in\Omega$, which we invoke as required:
\begin{gather}\label{ass:elliptic}
F(x,z)-\lambda (\mu^2+|z|^2)^\frac{p}{2} \text{ is convex in } z \tag{H1}\\[5pt]
\label{ass:growth} F(x,z)\leq \Lambda (1+|z|^q) \tag{H2}\\
\label{def:bounds3} |F(x,z)-F(y,z)|\leq \Lambda |x-y|^\alpha\left(1+ |z|^2\right)^\frac q 2.\tag{H3}
\end{gather}
Here $\mu\geq 0$, $\Lambda,\lambda>0$ and $\alpha\in(0,1]$. Moreover, we assume that $2\leq p<q<\infty$. If $p<n$, we further suppose $q\leq \frac{np}{n-p}$.

We remark that \eqref{ass:elliptic}, \eqref{ass:growth} are usually referred to as natural growth conditions.
Under these assumptions $F$ is convex in $z$ and (after adding a constant to $F$ if necessary) the following bounds apply for almost every $x\in\Omega$ and every $z,w\in \R^{N\times n}$:
\begin{gather}
\label{def:bounds1}(\mu^2+|z|^2+|w|^2)^\frac{p-2}{2}\lesssim\frac{F(x,z)-F(x,w)-\langle \partial_z F(x,w),z-w\rangle}{|z-w|^2}\tag{H4}\\[5pt]
\label{def:lower}F(x,z)\gtrsim |z|^p-1.\tag{H5}
\end{gather}

Let us also give a precise definition of the notions of minimisers we are interested in.
\begin{definition}
Given $\psi$, $g\in W^{1,q}(\Omega)=W^{1,q}(\Omega,\R^N)$, we say that
$u\in K^\psi_g(\Omega)$ is a \textbf{(pointwise) minimiser} of \eqref{obstacleProblem} if it holds that ${F(x,Du)\in L^1(\Omega)}$ and
\begin{align*}
\int_\Omega F(x,Du)\d x \leq \int_{\Omega} F(x,D\phi)\d x
\end{align*}
for all $\phi \in K^\psi_g(\Omega)$. 

Further $u\in K^\psi_g(\Omega)$ is a \textbf{relaxed minimiser} of \eqref{obstacleProblem} if $u$ minimises the relaxed functional
\begin{align*}
\overline \F(u) = \inf\Big\{\liminf_{j\to\infty} \int_\Omega F(x,\D u_j)\d x\colon u_j\in K^{*,\psi}_g(\Omega), u_j\rightharpoonup u \text{ weakly in } W^{1,p}(\Omega)\Big\}
\end{align*}
where
\begin{align*}
K^{*,\psi}_g(\Omega) = \{u\in W^{1,q}_g(\Omega)\colon u\geq \psi \text{ a.e. in } \Omega\},
\end{align*}
that is
\begin{align*}
\overline\F(u)\leq \overline\F(v) \text{ for all } v\in K^\psi_g(\Omega).
\end{align*}
\end{definition} 
\noindent We remark immediately that $K^{*,\psi}_g(\Omega)$ is dense in $K^\psi_g(\Omega)$ (see Corollary \ref{cor:density}) and that by weak lower semicontinuity of $\F(\cdot)$ in $W^{1,q}(\Omega)$, for $u\in K^{*,\psi}_g(\Omega)$, $\overline \F(u)=\F(u)$. Further using the direct method it is not difficult to establish the existence of both types of minimisers.

When no obstacle is present, the study of elliptic systems and functionals when $p=q$ is well established with a long list of important results. For an introduction and references, we refer to \cite{Giaquinta1983} and \cite{Giusti2003}. 

In the unconstrained case, the systematic study of regularity theory for minimisers when $p<q$ started with the seminal papers \cite{Marcellini1989,Marcellini1991}. We refer to \cite{Mingione2006,Mingione2021} for an overview of the theory and further references. 
We only list the, to our knowledge, best available $W^{1,q}_{\tp{loc}}$-regularity results for general autonomous convex functionals with $(p,q)$-growth when $n\geq 2$: under natural growth conditions it suffices to assume $q<\frac{np}{n-1}$ in order to obtain $W^{1,q}_{\tp{loc}}$-regularity of minimisers \cite{Carozza2013}. To obtain the same conclusion under controlled growth conditions, the gap may be widened to $q<p\left(1+\frac 2 {n-1}\right)$, see \cite{Schaeffner2020}, and under controlled duality growth conditions it suffices to take $q<\frac{np}{n-2}$ (if $n=2$ it suffices to take $q<\infty$) \cite{DeFilippis2020}. In all three cases, higher integrability goes hand in hand with a higher differentiability result. In the autonomous case our main result is new even in the unconstrained setting and gives a global analogue of the result in \cite{Carozza2013}. In the non-autonomous case the author in \cite{Koch2020} extended the work of \cite{Esposito2004,Fonseca2004} giving global versions of local $W^{1,q}$-regularity results obtained in \cite{Esposito2004,Fonseca2004} for relaxed minimisers of functionals $F(x,z)$, convex and with $(p,q)$-growth in $z$, while satisfying a uniform $\alpha$-H\"older condition in $x$ under the sharp assumption $q<\frac{(n+\alpha)p}{n}$. Under the additional assumption that

\begin{align}\label{def:changeOfXAlt}
\left\{\begin{array}{l} \textit{there is } \e_0>0 \textit{ such that for any } \e\in(0,\e_0) \text{ and } x\in\Omega  \textit{ there is } \\[5pt]
\hat y\in \overline{B_\e(x)\cap\Omega} \textit{ such that } F(\hat y,z)\leq F(y,z)
\textit{ for all } y\in \overline{B_\e(x)\cap\Omega}, z\in \R^{n\times m},
\end{array}\right.\tag{H6}
\end{align}
\noindent
first used in \cite{Esposito2019},
the results extend to pointwise minimisers. We note that
\eqref{def:changeOfXAlt} is very similar to Assumption 2.3 in \cite{Zhikov1995} and that \eqref{def:changeOfXAlt} holds for many of the commonly considered examples, see \cite{Esposito2004,Koch2020}, \cite{Esposito2019}. The $W^{1,q}_{\tp{loc}}$-regularity is in general not known if $q=\frac{(n+\alpha)p}{n}$. An exception are functionals modelled on the double-phase functional \cite{Baroni2018}. See also \cite{DeFilippis2019}.

Recently, constrained problems with $(p,q)$-growth have gained interest. Obstacle problems have a long history in their own right and we refer to \cite{Heinonen1993} for an introduction and further references with regards to the theory in the case of $p$-growth. When considering linear elliptic problems, the regularity of the solution agrees with the regularity of the obstacle \cite{Brezis1973,Caffarelli1980,Kinderlehrer1980}. In the non-linear setting, this is not the case and, usually, more regularity has to be assumed on the obstacle to overcome the effect of the non-linearity. One reason for the recent interest in obstacle problems with non-standard growth is that such problems have appeared in the construction of comparison problems for the study of fine properties of the solutions of certain non-linear PDEs with non-standard growth \cite{Chlebicka2020a,Fu2015,Kilpelainen2000,Heinonen1993}.

We remark that at the moment the theory in the vectorial setting and in the case of $(p,q)$-growth in particular is considerably less developed than that in the scalar setting, in particular when comparing to scalar quadratic growth. In the case of the scalar Laplacian, $F(x,z)=|z|^2$, the regularity of the free boundary $\p\{u>\psi\}$ has been studied in great detail \cite{Caffarelli1977,Caffarelli1998,Figalli2019,Figalli2020}. However already when considering the scalar $p$-Laplacian $F(x,z)=|z|^p$, considerably less is known \cite{Andersson2015,Figalli2017}. Finally we note that our constraint ${u\geq \psi}$ is a convex constraint. The case of non-convex constraints is harder, but nevertheless of interest, see e.g. the study of functionals with $(p,q)$-growth constrained to lie on a manifold \cite{DeFilippis2019a,DeFilippis2020d,Chlebicka2020}.

Returning to vectorial functionals with $(p,q)$-growth, until recently most of the focus has been on results concerning functionals with additional structure assumptions. These include improved integrability and differentiability results in the case of variable exponent functionals $\int_\Omega |\D u|^{p(x)}\d x$ \cite{Eleuteri2013,Fu2015,Eleuteri2018,Foralli2021}, as well as the double phase functional $\int_\Omega |\D u|^p+a(x)|\D u|^q\d x$ \cite{Byun2019,Chlebicka2020a,Zhang2021}. Another direction of research has been Cald\'{e}ron--Zygmund estimates for both double phase \cite{Byun2020} and variable exponent \cite{Liang2021} obstacle problems. We mention that improved integrability results are also available in the setting of almost linear growth \cite{Fuchs2000a,Ok2016} as well as certain parabolic settings, see for example \cite{Bogelein2012,Erhardt2016} for results and further references. As is usually observed in problems with non-standard growth, for radial integrands of the form $F(x,z)\equiv F(x,|z|)$, minimisers are better behaved and understood. In this setting, sharp assumptions giving Lipschitz regularity of minimisers can be found in \cite{deFilippis2020b} and integrands with generalised Orlicz-growth are studied in \cite{Karppinen2021}. Regularity results in the scale of Besov spaces in this setting were obtained in \cite{Grimaldi2021}

Obstacle problems for integrands satisfying $(p,q)$-growth, but no further structural assumptions, were first studied in \cite{DeFilippis2021a,Gavioli2019,Gavioli2019a} where local improved integrability and differentiability as well as local higher regularity properties of minimisers were established. The assumptions on the integrand in \cite{DeFilippis2021a} are similar to the ones that we assume in this paper. The results here, particularly regarding the non-autonomous case, may be regarded as global versions of those in \cite{DeFilippis2021a}. We remark, however, that our assumptions on the obstacle $\psi$ are stronger than those in \cite{DeFilippis2021a}. In particular, there it suffices to assume $\psi\in W^{1+\alpha,q}(\Omega)$ in the non-autonomous case whereas we require $\psi\in W^{2,\infty}(\Omega)$. In fact, under the assumption that $\psi\in W^{2,\infty}(\Omega)$, local H\"older-continuity of the minimiser is derived in \cite{DeFilippis2021a}. We do not prove an analogous global result here. Nevertheless, our assumption on the gap $q<\frac{(n+\alpha)p}{n}$ appears already in the unconstrained case and is sharp. We remark that a similar set-up has been studied (locally) in the non-autonomous case but with \eqref{def:bounds3} replaced by a Sobolev-dependence in \cite{Bertazzoni2021}.

The author believes the results in this paper to be the first global results on obstacle problems for vectorial functionals with $(p,q)$-growth but without further structure assumptions.
Our main theorem is the following:
\begin{theorem}\label{thm:main}
Suppose that $\Omega$ is a Lipschitz domain and the $C^2$ integrand $F\equiv F(z)$ satisfies \eqref{ass:elliptic} and \eqref{ass:growth} with  $2\leq p\leq q<\min(\frac{np}{n-1},p+1)$. Let $\psi\in W^{2,\infty}(\Omega)$ and $g\in W^{2,q}(\Omega)$. Then $u\in W^{1,q}(\Omega)$ where $u$ is the relaxed minimiser of \eqref{obstacleProblem}.

Suppose $F\equiv F(x,z)$ is $C^2$ in $z$ and satisfies \eqref{ass:elliptic}, \eqref{ass:growth} and \eqref{def:bounds3} with $2<p\leq q<\frac{(n+\alpha)p}{n}$. Suppose $\psi\in W^{2,\infty}(\Omega)$ and $g\in W^{1+\alpha,q}(\Omega)$. Then
$u\in W^{1,q}(\Omega)$ where $u$ is the relaxed minimiser of \eqref{obstacleProblem}.
\end{theorem}

To the best of our knowledge this is the first global $W^{1,q}$-regularity result valid for a large class of general convex $(p,q)$-growth functionals in the constrained case. The autonomous result is new even in the unconstrained case, and is the global equivalent of results in \cite{Carozza2013}. 

The proof of Theorem \ref{thm:main} is based on the difference quotient method. To apply the difference quotient method globally, we rely on an argument developed in \cite{Savare1998}, which the author also used in the unconstrained case in \cite{Koch2020}. However, the constraint causes additional difficulties, since all competitors need to satisfy the constraint. To overcome this issue we rely on an observation that is well known in the numerics literature. We refer to \cite{Tran2015,Schaeffer2018} for the necessary results in our set-up. For sufficiently large parameter $\kappa$, the $L^1$-penalisation of obstacle problems is exact, that is, minimisers $u$ of
\begin{align*}
W^{1,p}_g(\Omega)\ni v\rightarrow\int_\Omega F(x,\D v)+\kappa(\psi-v)_+\d x
\end{align*}
satisfy $u\geq \psi$ for $\kappa>\kappa_0(\|\psi\|_{W^{2,\infty}(\Omega)})$. 
The improvement in the autonomous case is due to a trick first used in \cite{Bella2020} and used in the context of vector-valued $(p,q)$-growth in \cite{Schaeffner2020}. Optimising the choice of cut-off function in dependence on the minimiser we can employ a Sobolev embedding at a crucial step of the proof on spheres instead of balls. Due to the dimensional dependence of the constants in the embedding this gives a stronger result, allowing to widen the gap from $q<\frac{(n+1)p}{n}$ to $q<\frac{np}{n-1}$.

We remark that Theorem \ref{thm:main} is phrased for relaxed minimisers. This is forced by the possible occurence of the Lavrentiev phenomenon, which describes the possibility that
\begin{align}\label{eq:Lphenomenon}
\inf_{u\in K^\psi_g(\Omega)} \F(u)< \inf_{u\in K^{*,\psi}_g(\Omega)} \F(u).
\end{align}
This phenomenon occurs in the unconstrained case. A first example of such behaviour was given in \cite{Lavrentiev1926}. In the context of $(p,q)$-growth functionals, the theory was further developed in \cite{Zhikov1987,Zhikov1993,Zhikov1995}. The Lavrentiev phenomenon is closely related to properties of the relaxed functional. We adopt the viewpoint and terminology of \cite{Buttazo1992}, and consider a topological space $X$ of weakly differentiable functions with a dense subspace $Y\subset X$. We introduce the following sequentially lower semi-continuous (slsc) envelopes:
\begin{align*}
\overline \F_X = \sup\{\,\mathscr G\colon X\to[0,\infty]: \mathscr G \text{ slsc}, \mathscr G\leq \F \text{ on } X\,\}\\
\overline \F_{Y} = \sup\{\,\mathscr G\colon X\to[0,\infty]: \mathscr G \text{ slsc}, \mathscr G\leq \F \text{ on } Y.\,\}\nonumber
\end{align*}
Then we define the Lavrentiev gap functional for $u\in X$ as
\begin{align*}
\mathscr L(u,X,Y)=\begin{cases}
		\overline \F_{Y}(u)-\overline \F_X(u)  &\text{ if } \overline \F_X(u)<\infty\\
		0 &\text{ else}.
	\end{cases}
\end{align*}
We note that the gap functional is non-negative. 

There is extensive literature on the Lavrientiev phenomenon and gap functional, an overview of which can be found in \cite{Buttazo1995,Foss2001}, as well as further references. The phenomenon is also of interest in non-linear elasticity \cite{Foss2003}. Considering the common choice ${X=W^{1,p}(\Omega)}$ endowed with the weak toplogy and $Y=W^{1,q}_{\tp{loc}}(\Omega)\cap W^{1,p}(\Omega)$, a question related to the Lavrentiev phenomenon is to study measure representations of $\overline \F(\cdot)$. We refer to \cite{Fonseca1997,Acerbi2003} for results and further references in this direction.

In this paper, we always consider the choice $X=K^\psi_g(\Omega)$ endowed with the weak topology inherited from $W^{1,p}(\Omega)$ and $Y=K^{*,\psi}_g(\Omega)$.
  Since $F(x,z)$ is convex, standard methods show that $\overline\F_X(\cdot)=\F(\cdot)$\cite[Chapter 4]{Giusti2003}. Furthermore, $\overline \F_Y(\cdot)=\overline \F(\cdot)$. We also note that if ${\mathscr L(u,X,Y)=0}$ for all $u\in X$, then the Lavrentiev phenomenon cannot occur. Non-occurrence of the Lavrentiev phenomenon allows us to transfer the estimates obtained in Theorem \ref{thm:main} to pointwise minimisers and thus to establish $W^{1,q}$-regularity.
  
In general, it is necessary to assume that $\mathscr L(u,K^\psi_g(\Omega),K^{*,\psi}_g(\Omega))=0$ for minimisers of \eqref{obstacleProblem} in order to replace relaxed minimisers in Theorem \ref{thm:main} with pointwise minimisers. Nevertheless, combining the arguments used in \cite{Koch2020} and \cite{Guerra2021b}, we are able to prove the following result:
  
\begin{proposition}\label{prop:pointwise}
Let $\alpha\in(0,1)$.
Suppose that $\Omega$ is a $C^{1,\alpha}$-domain and the assumptions of Theorem \ref{thm:main} hold for this choice of $\alpha$. In the non-autonomous case, assume additionally that \eqref{def:changeOfXAlt} holds and $q<p+1$.
 Then $u\in W^{1,q}(\Omega)$, where $u$ is the pointwise minimiser of \eqref{obstacleProblem}.
\end{proposition}

The structure of the paper is as follows. In Section \ref{sec:prelim}, we collect some background results. We present the proof of exactness of $L^1$-penalisation in Section \ref{sec:penalty}, before proving an apriori estimate for a regularised version of the $L^1$-penalised functional in Section \ref{sec:estimate}. This allow us to prove our in Section \ref{sec:mainproof} our main theorem Theorem \ref{thm:main}, as well as Proposition \ref{prop:pointwise}.

\section{Preliminaries}\label{sec:prelim}
\subsection{Notation}
In this section we introduce our notation. 
The set $\Omega$ always denotes a open, bounded domain in $\R^n$. Given a set $\omega\subset\R^n$, $\overline \omega$ denotes its closure. We write $B_r(x)$ for the usual open Euclidean ball of radius $r$ in $\R^n$ and $\S^{n-1}$ for the unit sphere in $\R^n$.
We denote the cone of height $\rho$, aperture $\theta$ and axis in direction $\pmb n$ by $C_\rho(\theta,\pmb n)$, that is, the set
\begin{align*}
C_\rho(\theta,\pmb n)=\{\,h\in \R^n: |h|\leq \rho, h\cdot \pmb n\geq |h|\cos(\theta)\,\}.
\end{align*}
Here $|\cdot |$ denotes the Euclidean norm of a vector in $\R^n$ and likewise the Euclidean norm of a matrix $A\in \R^{n\times n}$. The identity matrix in $\R^{n\times n}$ is denoted $\tp{Id}$.
Given an open set $\Omega$ we denote $\Omega_\lambda=\{\,x\in \Omega: d(x,\partial\Omega)>\lambda\,\}$ and $\lambda \Omega = \{\,\lambda x: x\in\Omega\,\}$, where $d(x,\partial\Omega)=\inf_{y\in \p\Omega} |x-y|$ denotes the distance of $x$ from the boundary of $\Omega$.

If $p\in[1,\infty]$ denote by $p'=\frac{p}{p-1}$ its H\"older conjugate.
The symbols $a \sim b$ and $a\lesssim b$ mean that there exists some constant $C>0$, depending only on $n,N,p,\Omega,\mu,\lambda$ and $\Lambda$, and independent of $a$ and $b$ such that $C^{-1} a \leq b \leq C a$ and $a\leq C b$, respectively. 

Write $V_{\mu,t}(z)=(\mu^2+|z|^2)^\frac{t-2}{4}z$. We recall the useful well-known inequality:
\begin{lemma}\label{lem:Vfunctional}
For every $s>-1$, $\mu\in[0,1]$, $z_1$, $z_2\in\R^N$, with $\mu+|z_1|+|z_2|>0$, we have
\begin{align*}
\int_0^1 (\mu^2+|z_1+\lambda(z_2-z_1)^2)^\frac s 2 \lambda\d \lambda\sim (\mu^2+|z_1|^2+|z_2|^2)^\frac s 2
\end{align*}
with the implicit constants only depending on $s$.
Furthermore,
\begin{align*}
|V_{\mu,t}(z_1)-V_{\mu,t}(z_2)|\sim (\mu^2+|z_1|^2+|z_2|^2)^\frac{t-2}{2}|z_1-z_2|^2.
\end{align*}
The implicit constants depend on $s$ and $N$ only.
\end{lemma}
We often find it useful to write for a function $v$ defined on $\R^n$ and a vector $h\in \R^n$, $v_h(x)=v(x+h)$.

We fix a family $\{\,\phi_\e\,\}$ of radially symmetric, non-negative mollifiers of unitary mass. We denote convolution with $\phi_\e$ as
\begin{align*}
u\star\phi_\e(x)=\int_{\R^n} u(y)\phi_\e(x-y)\d y.
\end{align*}

\subsection{Function spaces}
\label{sec:besov}
We recall some basic properties of Sobolev and Besov spaces following the exposition in \cite{Savare1998}. The theory can also be found in \cite{Triebel1978}.

For $0\leq \alpha\leq 1$ and $k\in \N$, $C^k(\Omega)$ and $C^{k,\alpha}(\Omega)$ denote the spaces of functions that are $k$-times continuously differentiable in $\Omega$ and $k$-times $\alpha$-H\"older differentiable in $\Omega$, respectively.

For $1\leq p\leq\infty$, $k\in \N$, we let $L^p(\Omega)=L^p(\Omega,\R^N)$ and $W^{k,p}(\Omega)=W^{k,p}(\Omega,\R^N)$ denote the usual Lebesgue and Sobolev spaces, respectively. We write $W^{k,p}_0(\Omega)$ for the closure of $C_0^\infty(\Omega)$-functions with respect to the $W^{k,p}$-norm. For $g\in W^{k,p}(\Omega)$, we write ${W^{k,p}_g(\Omega)= g + W^{k,p}_0(\Omega)}$. We freely identify $W^{k,p}$-functions with their precise representatives. 

 We denote by $[\cdot,\cdot]_{s,q}$ the real interpolation functor. Let $s\in(0,1)$ and $p$, $q\in [1,\infty]$. We define
\begin{gather*}
B^{s,p}_q(\Omega)=B^{s,p}_q(\Omega,\R^N)=[W^{1,p}(\Omega,\R^N),L^p(\Omega,\R^N)]_{s,q}\\[5pt]
B^{1+s,p}_q(\Omega)=[W^{2,p}(\Omega),W^{1,p}(\Omega)]_{s,q}=\{\,v\in W^{1,p}(\Omega): Dv\in B^{s,p}_q(\Omega)\,\}
\end{gather*}
Further, we recall that $W^{1+s,p}(\Omega)=B^{1+s,p}_p(\Omega)$ and that for $1\leq q<\infty$, $B^{s,p}_q(\Omega)$ embeds continuously in $B^{s,p}_\infty(\Omega)$.
We use a characterisation of these spaces in terms of difference quotients as follows:  let $D$ be a set generating $\R^n$, star-shaped with respect to $0$. For $s\in(0,1)$, $p\in[1,\infty]$, consider
\begin{align*}
[v]_{s,p,\Omega}^p := \sup_{h\in D\setminus\{\,0\,\}}\int_{\Omega_h}\left|\frac{v_h(x)-v(x)}{h}\right|^p\d x.
\end{align*}
This semi-norm characterises $B^{s,p}_\infty(\Omega)$ in the sense that
\begin{align*}
v\in B^{s,p}_\infty(\Omega)\Leftrightarrow v\in L^p(\Omega) \text{ and } [v]_{s,p,\Omega}^p<\infty.
\end{align*}
Moreover there are positive constants $C_1,C_2>0$ depending only on $s$, $p$, $D$, $\Omega$ such that
\begin{align}\label{eq:besovcharacterisation}
C_1 \|v\|_{B^{s,p}_\infty(\Omega)}\leq \|v\|_{L^p(\Omega)}+[v]_{s,p,\Omega}\leq C_2\|v\|_{B^{s,p}_\infty(\Omega)}.
\end{align}
If $\Omega=B_r(x_0)$, then the constants $C_1, C_2$ are unchanged by replacing $D$ with $QD$, where $Q$ is an orthonormal matrix. In particular, when $D=C_\rho(\theta,\pmb n)$ is a cone, they are independent of the choice of $\pmb n$. 

Finally, we recall that $B^{s,p}_q(\Omega)$ may be localised. If $\{\,U_i\,\}_{i\leq M}$ is a finite collection of balls covering $\Omega$, then $v\in B^{s,p}_q(\Omega)$ if and only if $v_{|\Omega \cap U_i}\in B^{s,p}_q(\Omega \cap U_i)$ for $i=1,...,M$. Moreover, there are constants $C_3$, $C_4$ such that
\begin{align}\label{eq:besovlocalisation}
C_3 \|v\|_{B^{s,p}_q(\Omega)}\leq \sum_{i=1}^M \|v\|_{B^{s,p}_q(\Omega \cap U_i)}\leq C_4 \|v\|_{B^{s,p}_q(\Omega)}.
\end{align}

We recall the following well-known embedding theorem; see, for example, \cite{Triebel2002}.
\begin{theorem}\label{thm:embedding}
Suppose that $\Omega$ is a Lipschitz domain. Let $0< s\leq 1$ and $p$, $p_1\in [1,\infty]$.
Assume that $s-\frac n p =-\frac n {p_1}$ and $v\in B^{s,p}_\infty(\Omega)$. Then, for any $\e\in(0,1-p_1]$, we have
\begin{align*}
\|v\|_{L^{p_1-\e}(\Omega)}\lesssim\|v\|_{B^{s,p}_\infty(\Omega)}.
\end{align*}
\end{theorem}

We have a trace theorem in the following form; see, for example, \cite{Evans1992}.
\begin{lemma}\label{lem:traceTheorem}
Let $\Omega$ be a Lipschitz domain and let $1<p<\infty$. There is a bounded linear operator $\tp{Tr}\colon W^{1+\frac 1 p,p}(\Omega)\to W^{1,p}(\p\Omega)$. Morever, $\tp{Tr}(u)$ may be defined to be the values of the precise representative of $u$ on $\p\Omega$.
\end{lemma}

Finally, we recall the following well-known result, which will justify extending $u\in W^{1,p}_g(\Omega)$ by extensions of $g$. We refer to \cite{Evans1992} for the ingredients of the proof.
\begin{lemma}\label{lem:extension}
Let $p\in[1,\infty]$ and $V\Supset \Omega$ an open, bounded set.
Suppose that $u\in W^{1,p}(\Omega)$ and $v\in W^{1,p}_u(\Omega)\cap W^{1,p}(V)$. Then the map
\begin{align*}
w = \begin{cases}
	u \text{ in } \Omega \\
	v \text{ in } V\setminus\Omega
	\end{cases}
\end{align*}
is an element of $W^{1,p}(V)$.
\end{lemma}

We require a Fubini-type theorem on spheres, that has independent interest. The statement is likely known to the expert, but we have been unable to find a reference in the literature. In the context of fractional Sobolev spaces such a result has been obtained in \cite{Gmeineder2021}. We remark that, while the proof in \cite{Gmeineder2021} relied on a geometric construction and direct calculation, our argument is based on interpolation.
\begin{lemma}\label{lem:FubiniType}
Let $\sigma>\rho\geq 0$ and denote $B_\sigma = B_\sigma(0)$, $B_\rho = B_\rho(0)$. 
Let $s\in(0,1)$ and ${1\leq \tau\leq q\leq p}$. If $v\in B^{s,p}_q(B_\sigma\setminus B_\rho)$, then we have
\begin{align*}
\int_\sigma^\rho \|v\|_{B^{s,p}_q(\p B_r)}^\tau\d r\lesssim \|v\|_{B^{s,p}_q(B_\sigma\setminus B_\rho)}^\tau,
\end{align*}
where the implicit constant depends only on $s,\tau,q,p, \sigma$ and $\rho$.
\end{lemma}
\begin{proof}
Let $v\in B^{s,p}_q(B_\sigma\setminus B_\rho)$.
Since $B^{s,p}_q(\p B_r)=[L^p(\p B_r),W^{1,p}(\p B_r)]_{s,q}$, we have
\begin{align*}
\|v\|_{B^{s,p}_q(\p B_r)} = \left(\int_0^\infty\left(t^{-s}K(v,t,r)\right)^q\frac{\d t}{t}\right)^\frac 1 q.
\end{align*}
where 
$$
K(v,t,r)=\inf\left\{ \|v_1\|_{L^p(\p B_r)} + t\|v_2\|_{W^{1,p}(\p B_r)}\colon v = v_1+v_2, \, (v_1,v_2)\in X\right\}
$$
for $X=L^p(\p B_r)\times W^{1,p}(\p B_r)$.
Using Jensen's inequality and Fubini's theorem, we find
\begin{align*}
\int_\rho^\sigma \|v\|_{B^{s,p}_q(\p B_r)}^\tau\d r=&\int_\rho^\sigma\left(\int_0^\infty (t^{-s} K(v,t,r))^q\d r \frac{\d t}{t}\right)^\frac \tau q\d r\\
\lesssim& \left(\int_0^\infty \int_\rho^\sigma (t^{-s} K(v,t,r))^q\d r \frac{\d t}{t}\right)^\frac \tau q\\
\lesssim& \left(\int_0^\infty t^{-s q} \int_\rho^\sigma K'(v,t,r)^q\d r \frac{\d t}{t}\right)^\frac \tau q= I
\end{align*}
where
\begin{align*}
K'(v,t,r)&=\inf\left\{ \|v_1\|_{L^p(\p B_r)} + t\|v_2\|_{W^{1,p}(\p B_r)}\colon v = v_1+v_2 \text{ in } B_\sigma\setminus B_\rho, (v_1,v_2)\in Y\right\}
\end{align*}
for $ Y = L^p(B_\sigma\setminus B_\rho)\times W^{1,p}(B_\sigma\setminus B_\rho).$
Using that $q\leq p$, we have 
\begin{align*}
\int_\rho^\sigma (\|v_1\|_{L^p(\p B_r)}+t\|v_2\|_{W^{1,p}(\p B_r)})^q\d r&\lesssim \int_\rho^\sigma \|v_1\|_{L^p(\p B_r)}^q+t^q\|v_2\|_{W^{1,p}(\p B_r)}^q\\
&\lesssim \|v_1\|_{L^p(B_\sigma\setminus B_r)}^q+ t^q \|v_2\|_{W^{1,p}(B_\sigma\setminus B_\rho)}^q\\
&\lesssim \left(\|v_1\|_{L^p(B_\sigma\setminus B_r)}+ t \|v_2\|_{W^{1,p}(B_\sigma\setminus B_\rho)}\right)^q.
\end{align*}
Thus, using Fatou's lemma,
\begin{align*}
I&\lesssim \left(\int_0^\infty t^{-s q}\inf\left(\|v_1\|_{L^p(B_\sigma\setminus B_r)}+ t \|v_2\|_{W^{1,p}(B_\sigma\setminus B_\rho)}\right)^q \d r \frac{\d t}{t}\right)^\frac \tau q
&= \|v\|_{B^{s,p}_q(B_\sigma\setminus B_\rho)}^\tau,
\end{align*}
where the infinimum is taken over $$\{(v_1,v_2)\in Y\colon v = v_1+v_2 \text{ in } B_\sigma\setminus B_\rho\}.$$
\end{proof}

We use the following interpolation inequality, which is a direct consequence of H\"older's inequality.
Given $p_1$, $p_2\in[1,\infty]$ and $0<\theta<1$, define $p_\theta$ by the identity $\frac 1 {p_\theta} = \frac \theta {p_1}+\frac {1-\theta}{p_2}$. Then, for $u\in L^{p_1}(\Omega)\cap L^{p_2}(\Omega)$, we have
\begin{align}\label{eq:interpolationInequality}
\|u\|_{L^{p_\theta}(\Omega)}\leq \|u\|_{L^{p_1}(\Omega)}^\theta \|u\|_{L^{p_2}(\Omega)}^{(1-\theta)}.
\end{align}

\subsection{Some properties of Lipschitz and \texorpdfstring{$C^{1,\alpha}$}-domains}
\label{sec:lipschitzDomains}
In this section we recall some properties of Lipschitz and $C^{1,\alpha}$-domains. For further details we refer to \cite{Grisvard1992}.
We say $\Omega\subset\R^n$ is a Lipschitz ($C^{1,\alpha}$) domain if $\Omega$ is an open subset of $\R^n$ and for every $x\in \p\Omega$, there exist a neighbourhood $V$ of $x$ in $\R^n$ and orthogonal coordinates $\{\,y_i\,\}_{1\leq i\leq n}$ such that the following holds:
\begin{enumerate}
\item $V$ is a hypercube in the new coordinates:
\begin{equation*}
V = \{\,(y_1,\ldots, y_n\,\}: -a_i<y_i<a_i,\, 1\leq i\leq n-1\,\}.
\end{equation*}
\item there exists a Lipschitz ($C^{1,\alpha}$) function $\phi$ defined in 
\begin{align*}
V'=\{\,(y_1,\ldots y_{n-1}): -a_i<y_i<a_i,\, 1\leq i\leq n-1\,\}
\end{align*}
such that
\begin{gather*}
|\phi(y')|\leq a_n/2 \text{ for every } y'=(y_1,\ldots, y_{n-1})\in V',\\[5pt]
\Omega\cap V = \{\,y = (y',y_n)\in V: y_n<\phi(y')\,\},\\[5pt]
\p\Omega \cap V = \{\,y=(y',y_n)\in V: y_n = \phi(y')\,\}.
\end{gather*}
\end{enumerate}

Let $\Omega$ be a Lipschitz domain. Then $\Omega$ satisfies a uniform exterior cone condition \cite[Section 1.2.2]{Grisvard1992}. Indeed, there exist $\rho_0$, $\theta_0>0$ and a map $\pmb n\colon \R^n\to \S^{n-1}$ such that for every $x\in \R^n,$
\begin{align}\label{eq:uniformCone}
C_{\rho_0}(\theta_0,\pmb n(x))\subset O_{\rho_0}(x)=\left\{\,h\in\R^n: |h|\leq \rho_0,\, \left(\Omega\setminus B_{3\rho_0}(x)\right)+h\subset\R^n\setminus\Omega\,\right\}.
\end{align}
Moreover there is a smooth vector field transversal to $\p\Omega$, that is, there exist $\kappa>0$ and ${X\in C^\infty(\R^n,\R^n)}$ such that
\begin{align*}
X\cdot \nu\geq \kappa \quad \text{ a.e. on } \p\Omega
\end{align*}
 where $\nu$ is the exterior unit normal to $\p\Omega$ \cite[Lemma 1.5.1.9.]{Grisvard1992}.

We recall the following Lemma from \cite{Koch2020} which enables us to to stretch a small neighbourhood of the boundary in a controlled manner. This is crucial in constructing sequences with improved integrability but unchanged boundary behaviour.
\begin{lemma}[Lemma 2.6 in \cite{Koch2020}]\label{lem:diffeos}
Suppose $\Omega$ is a $C^{1,\alpha}$-domain. Then
there is a family of domains $\Omega^s\Supset\Omega$ and a family of $C^{1,\alpha}$-diffeomorphisms $\Psi_s\colon \Omega^s\to\Omega$ such that
\begin{enumerate}
\item $J\Psi_s\to 1$ and $|D\Psi_s-\tp{Id}|\to 0$ uniformly in $\Omega^s$ as $s\nearrow 1$. Equivalently, $J\Psi_s^{-1}\to 1$ and $|D\Psi_s^{-1}-\tp{Id}|\to 0$ uniformly in $\Omega$ as $s\nearrow 1$.
\item If $g\in W^{1+\frac 1 q,q}(\Omega)$ there is an extension $\hat g$ of $g$ to $\Omega^s$ such that $\hat g\in W^{1,q}(\Omega^s)$ and $\hat g\circ \Psi_s^{-1}\in W^{1,q}_g(\Omega)$.
\end{enumerate}
\end{lemma}

We conclude this section by noting a number of extensions, that we may carry out if $\Omega$ is a Lipschitz domain.
Let $\Omega\Subset B(0,R)$.
From \cite{Rychkov1999}, if $g\in W^{s,p}(\Omega)$, there is an extension ${\tilde g\in W^{s,p}(\R^n,\R^N)}$ of $g$ with
\begin{align}\label{eq:Gextension}
\|\tilde g\|_{W^{s,p}(\R^n)}\lesssim \|g\|_{W^{s,p}(\Omega)}.
\end{align}
Furthermore, we can extend $F(x,z)$ to a function on $B(0,R) \times\R^{N\times n}$, still denoted $F(x,z)$, such that it satisfies
\begin{gather*}
|F(x,z)-F(y,z)|\lesssim \Lambda |x-y|^\alpha (1+|z|^2)^\frac q 2\\
|F(x,z)|\lesssim \Lambda(1+|z|^2)^\frac q 2.
\end{gather*}
We do this by setting, for $x\in B(0,R)\setminus\Omega$,
\begin{align*}
F(x,z)= \inf_{y\in\Omega}\left(F(y,z)+\Lambda \left(1+|z|^2\right)^\frac q 2 |x-y|\right).
\end{align*}

\subsection{Relaxed minimisers and Lavrentiev}
We begin by showing that $K^{*,\psi}_g(\Omega)$ is dense in $K^\psi_g(\Omega)$ with respect to the $W^{1,p}(\Omega)$ norm. Note that this is a requirement in order for us to make sense of the relaxed functional $\overline\F(\cdot)$.
\begin{lemma}\label{cor:density}
Suppose $g$, $\psi\in W^{1+1/q,q}(\Omega)$. Then $K^{*,\psi}_g(\Omega)$ is dense in $K^\psi_g(\Omega)$ with respect to the $W^{1,p}(\Omega)$ norm.
\end{lemma}
\begin{proof}
Let $u\in K^\psi_g(\Omega)$. We first show that without loss of generality we may assume that $u\in W^{1,q}(U)$ for some open neighbourhood $U$ of $\p\Omega$. Using the notation of Lemma \ref{lem:diffeos}, let $s\in(1/2,1)$ and extend $u$ by $\hat g$ and $\psi$ by $\hat \psi$ to $\Omega^s$. Define
\begin{align*}
u^s(x)= u(\Psi_s^{-1}(x)), \qquad \psi^s(x) = \psi(\Psi_s^{-1}(x))
\end{align*}
and consider $v^s = u^s-\psi^s+\psi$. Then we have $v^s\in K^\psi_g(\Omega)$ and clearly $v^s\to u$ in $W^{1,p}(\Omega)$ as $s\to 1$. Further $v\in W^{1,q}(U)$ for an open neighbourhood $U$ of $\p\Omega$.

Assuming now that $u\in W^{1,q}(U)$ for an open neighbourhood $U$ of $\p\Omega$, let $\eta\in C_c^\infty(\Omega)$ be a smooth cut-off function with $\eta=1$ in $\Omega\setminus U$. We then consider
\begin{align*}
u_\e = \eta (u\star \phi_\e-\psi\star \phi_\e+\psi)+(1-\eta) u.
\end{align*}
We note, because $u \geq \psi$, we have $u\star \phi_\e\geq \psi\star \phi_\e$. Thus $u_\e\in K^{*,\psi}_g(\Omega)$. Furthermore, $u_\e\to u$ in $W^{1,p}(\Omega)$ as $\e\to 0$. Hence the proof is complete.
\end{proof}

Given $F(x,z)$ satisfying $(p,q)$-growth, we consider the regularised functional
\begin{align*}
\F_\e(u)=\int_\Omega F(x,\D u)+\e|\D u|^q\d x.
\end{align*}
We wish to relate minimisers of $\overline\F(\cdot)$ to minimisers of $\F_\e(\cdot)$.

\begin{lemma}[c.f. Lemma 6.4. in \cite{Marcellini1989}]\label{lem:lavrentiev} Let $g$, $\psi\in W^{1+1/q,q}(\Omega)$. Suppose that $F(x,z)$ satisfies \eqref{ass:elliptic} and \eqref{ass:growth}.
We take $u$, a relaxed minimiser of $\F(\cdot)$ in the class $K^{\psi}_g(\Omega)$, and $u_\e$, the pointwise minimiser of $\F_\e(\cdot)$ in the class $K^{*,\psi}_g(\Omega)$. Then $\F_\e(u_\e)\to \overline\F(u)$ as $\e\to 0$. Moreover, up to a subsequence, we have $u_\e\to u$ in $W^{1,p}(\Omega)$.
\end{lemma}
\begin{proof}
Existence and uniqueness of $u_\e$ follows from the direct method and strict convexity, respectively.
We further note that 
\begin{align*}
\overline\F(u)\leq \liminf_{\e\to 0}\F(u_\e) \leq \liminf_{\e\to 0}\F_\e(u_\e).
\end{align*}
 To prove the reverse implication, we note that, for any ${v\in K^{*,\psi}_g(\Omega)}$,
\begin{align*}
\limsup_{\e \to 0} \F_\e(u_\e) \leq \lim_{\e \to 0} \F_\e(v)= \F(v)= \overline \F(v).
\end{align*}
By definition of $\overline\F(\cdot)$ the inequality above extends to all $v\in K^\psi_g(\Omega)$. In particular, it holds with the choice $v=u$. Thus $\F_\e(u_\e)\to \overline \F(u)$ as $\e\to 0$.

Using \eqref{def:bounds1}, we may extract a (non-relabelled) subsequence of $u_\e$ such that $u_\e\rightharpoonup v$ weakly in $W^{1,p}(\Omega)$ for some $v\in W^{1,p}(\Omega)$. Using our calculations above we see that $v$ is a relaxed minimiser of $\F(\cdot)$ in the class $K^\psi_g(\Omega)$. Using $\eqref{def:bounds1}$, it is easy to deduce that for every $w_1$, $w_2\in K^{*,\psi}_g$,
\begin{align}\label{eq:convexArg}
\overline \F\left(\frac{w_1+w_2}{2}\right)+\frac \nu p\|Dw_1-Dw_2\|_{L^p(\Omega)}^p\leq \frac 1 2\left(\overline \F(w_1)+\overline \F(w_2)\right).
\end{align}
Using the definition of $\overline \F(\cdot)$ and weak lower semicontinuity of norms, this estimate extends to ${w_1,\, w_2\in K^\psi_g}$. In particular, $\overline \F(\cdot)$ is convex and so $u=v$. Moreover the choice $w_1=u,\, w_2=u_\e$ in the estimate shows that $u_\e\to u$ in $W^{1,p}(\Omega)$.
\end{proof}

We comment that in the autonomous case, relaxed minimisers agree with pointwise minimiser if $g\in W^{1+1/q,q}(\Omega)$ (or alternatively $g\in W^{1,q}(\p\Omega)$) and $\psi\in W^{2,\infty}(\Omega)$. This is a direct consequence of the following observation. Some of the calculations are based on arguments in \cite{Guerra2021b}.
\begin{lemma}\label{lem:nonLavrentiev}
Suppose that $1<p\leq q<p+1$ and $F\equiv F(z)$ satisfies \eqref{ass:elliptic} and \eqref{ass:growth}. Let $\psi\in W^{2,\infty}(\Omega)$ and $g\in W^{1+1/q,q}(\Omega)$. Then, given $u\in K^\psi_g(\Omega)$, there exists a sequence of functions $u_k\in K^{*,\psi}_g$ such that $u_k\rightharpoonup u$ in $W^{1,p}(\Omega)$ and $\F(u_k)\to \F(u)$.
\end{lemma}
\begin{proof}
We utilise the notation and construction of Lemma \ref{cor:density}. We find
\begin{align*}
\int_\Omega F(\D v^s)\d x &\leq \int_\Omega F(\D u^s)\d x + \int_\Omega |\D(\psi^s-\psi)|(1+|\D\psi|+|D\psi^s|+|\D u^s|)^{q-1}\d x\\
&= A_1 + A_2.
\end{align*}
We have
\begin{align*}
A_1 = \int_\Omega |\tp{det}\Psi_s| F(\D u)\d x + \int_{\Omega^s\setminus\Omega} |\tp{det}\Psi_s| F(\D \hat g)|\d x \to \int_\Omega F(\D u)\d x<\infty
\end{align*}
as $s\to 1$. For the other term, as $q\leq p+1$ and using \eqref{ass:growth} and the mean value theorem, we have
\begin{align*}
A_2 \lesssim \|\psi\|_{W^{2,\infty}(\Omega)} (1+\|u^s\|_{W^{1,p}(\Omega)}^p+\|\psi^s-\psi\|_{W^{1,p}(\Omega)}^p) \to \|\psi\|_{W^{2,\infty}(\Omega)}(1+ \|u\|_{W^{1,p}(\Omega)}^p)
\end{align*}
Thus, by a version of dominated convergence, $\F(v^s)\to \F(u)$ as $s\to 1$. 
Hence, by a diagonal subsequence argument, as in Lemma \ref{cor:density} we may assume that $u\in W^{1,q}(U)$ for some open neighbourhood $U$ of $\p\Omega$. 

Considering $u_\e$ defind as in Lemma \ref{cor:density} we write
\begin{align*}
\int_\Omega F(\D u_\e)\d x = \int_{\{\eta = 0\}} F(\D u_\e)\d x+ \int_{\{\eta = 1\}} F(\D u_\e)\d x+ \int_{0<\eta<1} F(\D u_\e)\d x.
\end{align*}
In the region where $\eta =0$, we have $u_\e = u$. When $\eta =1$, using \eqref{ass:growth}, the mean value theorem and Jensen's inequality, we see that
\begin{align*}
&\int_{\{\eta = 1\}} F(\D u_\e)\d x \\
\leq& \int_{\{\eta=1\}} F(\D u \star \phi_\e)\d x+c \int_{\{\eta=1\}} |\D(\psi\star\phi_\e-\psi)|(1+|\D u\star\phi_\e|+|\D(\psi-\psi\star\phi_\e)|)^{q-1}\d x\\
\lesssim& \int_{\{\eta=1\}} F(\D u)\star \phi_\e\d x+\|\psi\|_{W^{2,\infty}(\Omega)}(1+\| u\star \phi_\e\|_{W^{1,p}(\Omega)}+\|\psi-\psi\star\phi_\e\|_{W^{1,p}(\Omega)})^p\\
\to& \int_{\{\eta=1\}} F(\D u)\d x+\|\psi\|_{W^{2,\infty}(\Omega)}(1+\| u\|_{W^{1,p}(\Omega)})^p.
\end{align*}
as $\e\to 0$.
By a version of dominated convergence, we deduce that
\begin{align*}
\int_{\{\eta = 1\}} F(\D u_\e)\d x\to \int_{\{\eta=1\}} F(\D u)\d x.
\end{align*}
as $\e\to 0$. It remains to deal with the region where $\eta\in(0,1)$. This region is contained in $U$ and so the only thing we need to check is that we are able to apply dominated convergence. Using \eqref{ass:growth}, we estimate
\begin{align*}
&\int_{\{0<\eta<1\}}F(\D u_\e)\d x\\
 \lesssim& \int_U 1+|\D (u\star \phi_\e-\psi\star\phi_\e+\psi)|^q+|D\eta|^q|u\star\phi_\e-\psi\star\phi_\e+\psi|^q\d x\\
\leq& \int_U 1+|\D (u\star \phi_\e-\psi\star\phi_\e+\psi)|^q+\|D\eta\|_{L^\infty(\Omega)}^q|u\star\phi_\e-\psi\star\phi_\e+\psi|^q\d x\\
\to& \int_U 1+|\D u|^q+\|D\eta\|_{L^\infty(U)} |u|^q\d x<\infty.
\end{align*}
Thus the proof is complete.
\end{proof}

Now we turn to the non-autonomous case, where we prove a similar statement to Lemma \ref{lem:nonLavrentiev}, as long as \eqref{def:changeOfXAlt} is satisfied. We begin by recalling the following key lemma from \cite{Esposito2019}, which we use to prove Lemma \ref{lem:nonlavrentievNonAuton}.
\begin{lemma}\label{lem:esposito}
Assume that $1<p\leq q\leq \frac{(n+\alpha)p}{n}$. Let $\Omega$ be a domain and $u\in W^{1,p}(\Omega)$. Suppose that $F(x,\cdot)$ satisfies \eqref{ass:elliptic}, \eqref{ass:growth}, \eqref{def:bounds3} and \eqref{def:changeOfXAlt}. 
Then, for $x\in\Omega$ and with ${\e\leq \min(\e_0,d(x,\p\Omega))}$,
\begin{align*}
F(x,Du(\cdot)\star \phi_\e(x))\lesssim 1 + \big(F(\cdot,Du(\cdot))\star \phi_\e\big)(x).
\end{align*}
\end{lemma}

\begin{lemma}\label{lem:nonlavrentievNonAuton}
Let $\alpha\in(0,1)$. Suppose that $F(x,z)$ satisfies \eqref{ass:elliptic}, \eqref{ass:growth}, \eqref{def:bounds3} and \eqref{def:changeOfXAlt} and that $1<p\leq q<\min(p+1,\frac{(n+\alpha)p}{n})$. Let $\psi\in W^{2,\infty}(\Omega)$ and $g\in W^{1+1/q,q}(\Omega)$. Given $u\in K^\psi_g(\Omega)$ there exists a sequence of functions $u_k\in K^{*,\psi}_g$ such that $u_k\rightharpoonup u$ in $W^{1,p}(\Omega)$ and a sequence of integrands $F^k(x,z)$ satisfying \eqref{ass:elliptic}, \eqref{ass:growth}, \eqref{def:bounds3} and \eqref{def:changeOfXAlt}, with bounds independent of $k$, such that $\int_\Omega F^k(x,\D u^k)\d x\to \F(u)$ as $k\to\infty$.
\end{lemma}
\begin{proof}
We proceed along the same lines as in Lemma \ref{lem:nonLavrentiev}. Let $s\in(1/2,1)$ and extend $u$ by $\hat g$ and $\psi$ by $\hat \psi$ to $\Omega^s$. Using the notation of Lemma \ref{lem:diffeos}, we then define
\begin{align*}
u^s(x)= u(\Psi_s^{-1}(x)), \qquad \psi^s(x) = \psi(\Psi_s^{-1}(x))\\
F^s(x,z)=F\left(\Psi_s^{-1}(x),z D\Psi_s(\Psi_s^{-1}(x))\right),
\end{align*}
and denote $v^s = u^s-\psi^s+\psi$. We have $v^s\in K^\psi_g(\Omega)$ and clearly $v^s\to v$ in $W^{1,p}(\Omega)$ as $s\to 1$. It is straightforward to check that $F^s(x,z)$ satisfies the required assumptions and we can argue exactly as in Lemma \ref{lem:nonLavrentiev}
to see that it suffices to find $(u_j)\subset K^{*,\psi}_g(\Omega)$ such that $\int_\Omega F^s(x,\D u_j)\d x\to \int_\Omega F^s(x,\D u)\d x$ as $j\to\infty$.

 Let $\eta\in C_c^\infty(\Omega)$ be a smooth cut-off function with $\eta=1$ in $\Omega\setminus U$. We then consider
\begin{align*}
u_\e = \eta (u\star \phi_\e-\psi\star \phi_\e+\psi)+(1-\eta) u.
\end{align*}
The estimate proceeds exactly as in Lemma \ref{lem:nonLavrentiev}, the only difference being the estimate in the region where $\eta =1$. Here, we argue using \eqref{ass:growth}, the mean value theorem, Jensen's inequality and Lemma \ref{lem:esposito}, to see that
\begin{align*}
&\int_{\{\eta = 1\}} F^s(x,\D u_\e)\d x \\
\leq& \int_{\{\eta=1\}} F^s(x,\D u \star \phi_\e)\d x\\
&\quad+c \int_{\{\eta=1\}} |\D(\psi\star\phi_\e-\psi)|(1+|\D u\star\phi_\e|+|\D(\psi-\psi\star\phi_\e)|)^{q-1}\d x\\
\lesssim& \int_{\{\eta=1\}}1+ F(\cdot,\D u(\cdot))\star \phi_\e(x)\d x\\
&\quad+\|\psi\|_{W^{2,\infty}(\Omega)}(1+\|u\star \phi_\e\|_{W^{1,p}(\Omega)}+\|\psi-\psi\star\phi_\e\|_{W^{1,p}(\Omega)})^p\\
\to& \int_{\{\eta=1\}} 1+F^s(x,\D u)\d x+\|\psi\|_{W^{2,\infty}(\Omega)}(1+\| u\|_{W^{1,p}(\Omega)})^p.
\end{align*}
as $\e\to 0$.
By a version of dominated convergence, we deduce that
\begin{align*}
\int\Omega F^s(x,\D u_\e)\d x\to \int_\Omega F^s(x,\D u)\d x,
\end{align*}
as $\e\to 0$. 
\end{proof}

\section{Exactness of \texorpdfstring{$L^1$}{}-penalisation}\label{sec:penalty}
In this section we prove that $L^1$-penalisation of the obstacle problem \eqref{obstacleProblem} is exact. A version of the result can be found in \cite{Schaeffer2018} but we present a full argument for completeness. We consider the functional
\begin{align}\label{L1Obstacle}
\tilde\F(v)=\int_\Omega F(\D v)+\kappa |(\psi-v)_+|\d x.
\end{align}
and prove that for sufficiently large $\kappa>0$, pointwise minimisers $\tilde u$ of $\tilde \F(\cdot)$ satisfy $\tilde u\geq \psi$ a.e. in $\Omega$. Here and throughout this section, we denote ${(\psi-\tilde u)_+ = ((\psi_i-\tilde u_i)_+)}$.  In particular, pointwise minimisers of \eqref{obstacleProblem} and \eqref{L1Obstacle} agree. Consequently, we will prefer to work with the unconstrained functional $\tilde \F(\cdot)$ rather than the original functional $\F(\cdot)$.

\begin{proposition}[c.f.\cite{Schaeffer2018}]\label{prop:exact}
Suppose that $\Omega$ is a Lipschitz domain.
Let $\psi\in W^{2,\infty}(\Omega)$ and $g\in W^{1,p}(\Omega)$ be given such that $\psi\geq g$ on $\p\Omega$ in the sense of traces. Suppose ${F\equiv F(x,z)\colon \Omega\times \R^{N\times n}\to \R}$ is measurable in $x$ and $C^2$ in $z$, \eqref{ass:elliptic} holds and there exists $u_0\in W^{1,p}_g(\Omega)$ such that $\int_\Omega F(x,\D u_0)\d x<\infty$. Then there exists $\kappa_0=\kappa_0(\|\psi\|_{W^{2,\infty}(\Omega)})$ such that, for $\kappa>\kappa_0$, the minimiser $\tilde u$ of the functional $\tilde \F(\cdot)$
in the class $W^{1,p}_g(\Omega)$ satisfies $u\geq \psi$ a.e. in $\Omega$.
\end{proposition}
\begin{proof}
Fix $\kappa>0$, the value of which is yet to be determined.
By the direct method, a minimiser $\tilde u$ of $\tilde F(v)$ in the class $W^{1,p}_g(\Omega)$ exists and is unique.

Introduce
\begin{align*}
w = \tilde u+ (\psi-\tilde u)_+.
\end{align*}
Note that $w\in W^{1,p}_g(\Omega)$ as $\tilde u \geq \psi$ on $\p\Omega$ in the sense of traces. Thus, using \eqref{def:bounds1}, we see that
\begin{align*}
\tilde\F(w)= \int_\Omega F(x,\D w)\d x \leq& \int_\Omega F(x,\D \tilde u) - \p_z F(x,\D w)\cdot \D(x, \tilde u-w)\d x\\
=& \int_\Omega F(x,\D \tilde u)+\p_z F(x,\D w)\cdot \D((\psi-\tilde u)_+)\d x = A
\end{align*}
However, $\D(\psi-\tilde u)_+=0$ on $\{\psi\leq \tilde u\}$ and $\D(\psi-\tilde u)_+= \D \psi-\D \tilde u$ on $\{\psi\geq \tilde u\}$. Furthermore, $w=\psi$ on $\{\psi \geq \tilde u\}$. Hence, using integration by parts, we find that
\begin{align*}
A &= \int_\Omega F(x,\D \tilde u)+\p_z F(x,\D\psi)\cdot \D((\psi-\tilde u)_+)\d x\\
&= \int_\Omega F(x,\D \tilde u)-\tp{div}( \p_z F(x,\D \psi))(\psi-\tilde u)_+\d x\\
&\leq \int_\Omega F(x,\D \tilde u)+\|\tp{div}(\p_z F(x,D\psi))\|_{L^\infty(\Omega)}\int_\Omega |(\psi-\tilde u)_+|\d x.
\end{align*}
Thus if $\kappa > \kappa_0=\|\tp{div}(\p_z F_\e(\D\psi))\|_{L^\infty(\Omega)}$ the claim follows.
\end{proof}

\begin{cor}
Suppose that the hypothesis of Proposition \ref{prop:exact} hold. Then, if $\kappa>\kappa_0$, we have
\begin{align*}
\inf_{u\in K^\psi_g(\Omega)} \F(u) = \inf_{u\in W^{1,p}_g(\Omega)} \tilde \F(u).
\end{align*}
Further minimisers exist and agree.
\end{cor}

\section{An apriori estimate for a regularised functional}\label{sec:estimate}
A drawback of working with \eqref{L1Obstacle} is that the integrand is not differentiable. This causes issues in applying the Euler-Lagrange equation and monotonicity inequalities, such as \eqref{def:lower}. This motivates us to introduce the following regularisation.

Fix $H_\delta\in C^\infty(\R)$ such that $H_\delta(x)\to \max(0,x)$ uniformly in $\R$, and $H_\delta(\cdot)$ is non-negative, convex, non-decreasing, with $\|H'_\delta\|_{L^\infty(\R)}\leq 2$ for $\delta\leq 1$.
Given $\G(\cdot)$ satisfying \eqref{ass:elliptic} and \eqref{ass:growth} with $p=q$, we define for $\delta\in(0,1)$,
\begin{align}\label{regObstacle}
\tilde\G_\delta(w)=\int_\Omega G(x,\D w)+\kappa H_\delta\left((H_\delta(\psi_i-w_i))\right)\d x.
\end{align}
By the direct method, minimisers $\tilde u_\delta$ of $\tilde\G_\delta(\cdot)$ in the class $W^{1,p}_g(\Omega)$ exist and are unique.
Let $\tilde u$ be the minimiser of the functional $\tilde\G(\cdot)$ defined through \eqref{L1Obstacle} posed for $\G(\cdot)$.
Then $\tilde u$ is a valid competitor in the problem \eqref{regObstacle}. Hence if $\tilde u_\delta$ minimises \eqref{regObstacle}, we have
\begin{align*}
\int_\Omega G(x,\D \tilde u_\delta)+\kappa H_\delta\left((H_\delta(\psi_i-\tilde u_{\delta,i}))\right)\d x&\leq \int_\Omega G(x,\D \tilde u)+\kappa H_\delta\left((H_\delta(\psi_i-\tilde u_i))\right)\d x\\
&\to \int_\Omega G(x,\D \tilde u)+\kappa(\psi-\tilde u)_+\d x<\infty,
\end{align*}
as $\delta \to 0$.
In particular, using \eqref{def:lower}, we extract a subsequence, not relabelled, such that $\tilde u_\delta\rightharpoonup v$ weakly to some $v\in W^{1,p}_g(\Omega)$. Necessarily, $v$ is a minimiser of $\tilde \G(\cdot)$. In fact, for any $\delta>0$, we have
\begin{align*}
\int_\Omega G(x,\D \tilde u_\delta)+\kappa H_\delta\left((H_\delta(\psi_i-\tilde u_{\delta,i}))\right)\d x\leq \int_\Omega G(x,\D \tilde u)+\kappa H_\delta\left((H_\delta(\psi_i-\tilde u_i))\right)\d x.
\end{align*}
Taking limits as $\delta\to 0$ on both sides the claim follows. Since $\tilde G(\cdot)$ is convex, we deduce that $v=u$.
Moreover, due to \eqref{ass:elliptic} and the convexity of $H_\delta(\cdot)$, we have
\begin{align*}
&\quad\int_\Omega G(x,\D (\tilde u_\delta+\tilde u)/2)+\kappa H_\delta\left((H_\delta(\psi_i-(\tilde u_{\delta,i}+\tilde u_i)/2))\right)\d x+c\|\tilde u-\tilde u_\delta\|_{W^{1,p}(\Omega)}\\
&\leq \frac 1 2\left(\int_\Omega G(x,\D \tilde u_\delta)+\kappa H_\delta\left((H_\delta(\psi_i-\tilde u_{\delta,i}))\right)+G(x,\D \tilde u)+\kappa H_\delta\left((H_\delta(\psi_i-\tilde u_i))\right)\d x\right).
\end{align*}
Letting $\delta\to 0$ and combining that $\tilde u_\delta\to u$ strongly in $L^p(\Omega)$ with the lower semi-continuity of $G(\cdot)$ in $W^{1,p}(\Omega)$, we deduce that $\tilde u_\delta \to u$ strongly in $W^{1,p}(\Omega)$.

Recalling Lemma \ref{lem:lavrentiev} and taking a diagonal subsequence, the key tool we require in order to prove our main theorem is an apriori bound for minimisers of the functional
\begin{align*}
\F_{\e,\delta}(u)=\int_\Omega F(x,\D u)+\e|\D u|^q+\kappa H_\delta\left((H_\delta(\psi_i-u_i))\right)\d x,
\end{align*}
where the bound is independent of $\delta$ and $\e$.

\subsection{The autonomous case}
In this section we consider minimisers of $\F_{\e,\delta}(\cdot)$ in the autonomous case, that is, when $F(x,z)=F(z)$, and derive an apriori estimate that is independent of $\e$ and $\delta$.

Our goal is to prove the following.
\begin{proposition}\label{prop:auton}
Suppose $2\leq p\leq q<\min\left(p+1,\frac{np}{n-1}\right)$.
Let the data $g\in W^{2,q}(\Omega)$ and ${\psi\in W^{2,\infty}(\Omega)}$ be given. Suppose that $F\equiv F(z)$ satisfies \eqref{ass:elliptic} and \eqref{ass:growth}. Then, for minimisers $u_{\e,\delta}$ of $\F_{\e,\delta}(\cdot)$, for any $\alpha\in(0,1/2)$ and some $\beta>0$, we have that
\begin{align*}
\|\D u_{\e,\delta}\|_{L^q(\Omega)}^q&\lesssim[V_{\mu,p}(\D u_{\e,\delta})]_{B^{\alpha,2}_\infty(\Omega)}^2\\
&\lesssim \left(1+\|V_{\mu,p}(\D u_{\e,\delta})\|_{L^2(\Omega)}^2+\|g\|_{W^{2,q}(\Omega)}^q\right)^\frac{\beta} {1-\beta}\left(1+\int_\Omega F_\e(\D u_{\e,\delta})\d x\right)^\frac{1}{1-\beta}
\end{align*}
where the implicit constant is independent of $\e,\delta$.
\end{proposition}

The key tool in the proof is the following Lemma. It is a version of the key lemma in \cite{Schaeffner2020} but adapted to our purposes.
\begin{lemma}\label{lem:schaffner}
Fix $n\geq 2$ and let $t>1$. For given $0<\rho<\sigma<\infty$ with $\sigma-\rho<1$ and $w\in L^1(B_\sigma)$, consider
\begin{align*}
J(\rho,\sigma,w) = \inf \left\{\int_{B_\sigma} |w| (|\D \phi|+|\D \phi|^t)\d x\colon \phi\in C^1_0(B_\sigma),\, \phi \geq 0,\, \phi =1 \text{ in } B_\rho\right\}.
\end{align*}
Then, for every $\delta\in(0,1)$, we have
$$
J(\rho,\sigma,w)\leq (\sigma-\rho)^{-t-1/\delta}\left(\int_\rho^\sigma \left(\int_{\p B_r} |w|\d \sigma\right)^\delta \d r\right)^\frac {1} {\delta}.
$$
Moreover, given $\e>0$, if $|w|\geq 1$ in $B_\sigma$, there exists a radial symmetric $\tilde\phi\in C^1_0(B_\sigma)$ with $\tilde \phi\geq 0$ and $\tilde \phi =1$ in $B_\sigma$ such that
$$\int_{B_\sigma} |w| (|\D \tilde\phi|+|\D \tilde\phi|^t)\d x\leq (\sigma-\rho)^{-t-1/\delta}\left(\int_\rho^\sigma \left(\int_{\p B_r} |w|\d \sigma\right)^\delta\d r\right)^\frac {1} {\delta}+\e,$$
and, for almost every $\rho\leq r_1<r_2\leq \sigma$,
$$
\frac{|\tilde\phi(r_2)-\tilde\phi(r_1)|}{(r_2-r_1)}\leq (\sigma-\rho)^{-t-1/\delta}\left(\int_\rho^\sigma \left(\int_{\p B_r} |w|\d \sigma\right)^\delta\d r\right)^\frac {1} {\delta}+\e.
$$

\end{lemma}
\begin{proof}
The estimate follows by considering appropriate radially symmetric cut-off functions. Indeed, for any $\e\geq 0$,
\begin{align*}
&\quad J(\rho,\sigma,w)\\
&\leq \inf\left\{\int_\rho^\sigma (|\phi'|+|\phi'|^t)\left(\int_{\p B_r} |w|+\e\d \sigma\right)\d r\colon \phi\in C^1([\rho,\sigma]),\,\phi(\rho)=1,\,\phi(\sigma)=0\right\}\\
&=J_{1d,\e}.
\end{align*}
Since $w\in L^1(B_\sigma)$, by employing a standard approximation argument we can replace ${\phi\in C^1([\rho,\sigma])}$ with $\phi\in W^{1,\infty}(\rho,\sigma)$ in the definition of $J_{1d,\e}$. Then we set
\begin{align*}
\tilde \phi(r)=1-\left(\int_\rho^\sigma b(r)^{-1}\d r\right)^{-1}\left(\int_{\rho}^r b(r)^{-1}\d r\right) \quad\text{ where } b(r)=\int_{\p B_r} |w|\d\mathscr H^{n-1}+\e.
\end{align*}
Clearly $\tilde \phi\in W^{1,\infty}(\rho,\sigma)$ with $\tilde \phi(\rho)=1$ and $\tilde \phi(\sigma)=0$. Hence, we have that
\begin{align*}
J_{1d,\e}\leq \int_\rho^\sigma (|\tilde \phi'|+|\tilde \phi'|^t) b(r)\d r = \dfrac{\sigma-\rho}{\int_\rho^\sigma b(r)^{-1}\d r}+\frac{\int_\rho^\sigma b(r)^{1-t}\d r}{\left(\int_\rho^\sigma b(r)^{-1}\d r\right)^t}= I + II
\end{align*}
By H\"older's inequality, for any $s>1$, we have
\begin{align*}
(\sigma-\rho)= \int_\rho^\sigma \left(\frac{b(r)}{b(r)}\d r\right)^\frac{s-1}{s} \leq\left(\int_\rho^\sigma b(r)^{s-1}\d r\right)^{\frac{1}{s}} \left(\int_\rho^\sigma b(r)^{-1}\d r\right)^\frac{s-1}{s},
\end{align*}
Thus, for any $\delta>0$,
\begin{align*}
I\leq (\sigma-\rho)^{-(1+\delta)/\delta}\left(\int_\rho^\sigma\left(\int_{\p B_r} |w|\d \mathscr H^{n-1}+\e\d r\right)^\delta\right)^\frac 1 \delta.
\end{align*}
We now focus on the second term $II$. Using Jensen's inequality, we see that
\begin{align*}
\int_\sigma^\rho b(r)^{1-t}\d r\leq(\sigma-\rho)^{2-t} \left(\int_\sigma^\rho b(r)^{-1}\d r\right)^{t-1}.
\end{align*}
Hence, for any $\delta>0$, using the estimate for the term $I$, we obtain
\begin{align*}
II&\leq (\sigma-\rho)^{2-t} \left(\int_\sigma^\rho b(r)^{-1}\d r\right)^{-1}\\
&\leq (\sigma-\rho)^{1-t-(1+\delta)/\delta} \left(\int_\rho^\sigma b(r)^\delta\d r\right)^{1/\delta}
\end{align*}
Collecting estimates and letting $\e\to 0$ the result follows.

For the latter statement in the lemma, since $|w|\geq 1$, we note that
$$
 \frac{|\tilde \phi(r_1)-\tilde \phi(r_2)|}{r_2-r_1} = \frac{\int_{r_1}^{r_2} b(r)^{-1}\d r}{(r_2-r_1)\int_\rho^\sigma b(r)^{-1}\d r}
\leq \frac{1}{\int_\rho^\sigma b(r)^{-1}\d r}.
$$
The claim follows, estimating as for the term $I$.
\end{proof}

We note the following consequence of the previous Lemma. We set $\xi = \frac{n-1}{n-2}$ if $n>2$ and $\xi = \infty$ if $n=2$. Suppose that $w\in L^q(B_\sigma)$, $t>1$ and $q<\frac{np}{n-1}$. Then, for any $\delta\in(0,1]$ and $\rho<\sigma$ with $|\rho-\sigma|<1$,
\begin{align*}
J(\sigma,\rho,|w|^q)\leq (\sigma-\rho)^{-t-1/\delta}\left(\int_\rho^\sigma\|w\|_{L^q(\p B_r)}^{q\delta}\d r\right)^\frac 1 \delta.
\end{align*}
By H\"older's inequality, we find that
\begin{align*}
\left(\int_\rho^\sigma\|w\|_{L^q(\p B_r)}^{q\delta}\d r\right)^\frac 1 \delta&\leq \left(\int_\rho^\sigma \|w\|_{L^p(\p B_r)}^{\theta q \delta}\|w\|_{L^{\xi p}(\p B_r)}^{(1-\theta)q\delta}\d r\right)^\frac 1 \delta\\
&\leq \left(\int_\rho^\sigma \|w\|_{L^p(\p B_r)}^{\theta q\delta s/(s-1)}\d r\right)^\frac{s-1}{s\delta} \left(\int_\rho^\sigma \|w\|_{L^{\xi p}(\p B_r)}^{(1-\theta) q \delta s}\d r\right)^\frac{1}{\delta s}
\end{align*}
where $\theta\in(0,1)$ is such that
\begin{align*}
\frac \theta p + \frac{1-\theta}{\xi p } = \frac 1 q, \qquad s>1.
\end{align*}
We make the admissible choice
\begin{align*}
\delta = \frac p q, \qquad s = \frac 1 {1-\theta}.
\end{align*}
It follows that
\begin{align}
J(\sigma,\rho,|w|^q)\leq (\sigma-\rho)^{-t-q/p}\left(\int_{B_\sigma\setminus B_\rho} |w|^p\d r\right)^{\frac{\xi}{\xi-1}(1-\frac q {\xi p})}\left(\int_\rho^\sigma \|w\|_{L^{\xi p}(\p B_r)}^p\d r\right)^{\frac \xi {\xi-1} \left(\frac q p-1\right)}.
\end{align}
Here we understand that $\frac{\infty}{\infty-1} = 1$.

By the same argument, if $|w|\geq 1$ in $B_\sigma$ and $\e>0$, we can pick a radial test function $\tilde \phi$ such that, for almost every $\rho\leq r_1<r_2\leq \sigma$,
\begin{align}\label{eq:basicRadial}
&\quad\int_{B_\sigma} |w|^q (|\D \tilde\phi|+|\D \tilde\phi|^p)\d x+\frac{|\tilde\phi(r_2)-\tilde\phi(r_1)|}{|r_2-r_1|}\nonumber\\
&\leq (\sigma-\rho)^{-p-q/p}\left(\int_{B_\sigma\setminus B_\rho} |w|^p\d r\right)^{\frac{\xi}{\xi-1}(1-\frac q {\xi p})}\left(\int_\sigma^\rho \|w\|_{L^{\xi p}(\p B_r)}^p\d r\right)^{\frac \xi {\xi-1} \left(\frac q p-1\right)}+\e.
\end{align}
This statement is exactly the result we need and so we can proceed to prove the main estimate of this section.
\begin{proof}[Proof of Proposition \ref{prop:auton}]
By the direct method and strict convexity deriving from \eqref{ass:elliptic}, we obtain  the existence of minimisers $v_{\e,\delta}\in W^{1,q}_0(\Omega)$ that solve
\begin{align*}
\min_{v\in W^{1,q}_0(\Omega)} \F_{\e,\delta}(v+g)
\end{align*}
Moreover, $v_{\e,\delta}$ satisfies the Euler--Lagrange equation
\begin{align}\label{eq:EulerL}
\int_\Omega \partial_z F_\e(\D v_{\e,\delta}+\D g)\cdot \D \phi+\p_y \tilde H_\delta(v_{\e,\delta}+g)\cdot \phi\d x=0 \qquad\forall\phi\in W^{1,q}_0(\Omega).
\end{align} 
Here $\tilde H_\delta(u)=H_\delta\left(((H_\delta(\psi_i-u_i))\right)$.
For notational simplicity we suppress the dependence on $\e$ and $\delta$ for the time being and write $v= v_{\e,\delta}$.

Let $\rho_0>0$ and $\pmb n\colon \R^n\to \S^{n-1}$ be such that the uniform cone property \eqref{eq:uniformCone} holds. Without loss of generality, we may assume that $\rho_0<1$.
Let $\rho_0\geq\sigma>\rho\geq\rho_0/2$. Fix $x_0\in \Omega$ and take $\phi\in C^1_c(\R^n)$, a radially symmetric, monotonic decreasing function with $\phi=1$ in $B_\rho(x_0)$, $\supp\, \phi\subset B_{\sigma}(x_0)$. We denote the extension of $v$ by $0$ to $\R^n$ by $\tilde v$ and write $T_h \D \tilde v = \phi \D \tilde v_h+(1-\phi)\D \tilde v$ and $T_h v = \phi \tilde v_h + (1-\phi)\tilde v$. Using \eqref{ass:elliptic}, the convexity of $\tilde H_\delta(\cdot)$ and \eqref{eq:EulerL} for $h\in C_{\rho_0/4}(\theta_0,\pmb n(x_0))$, we see that
\begin{align}\label{eq:lowerEstimateAuton}
&\quad\int_\Omega F_\e(\D T_h \tilde v+\D  g)+\tilde H_\delta(T_h \tilde v+g)-F_\e(\D\tilde v+\D g)-\tilde H_\delta(\tilde v+g)\d x\nonumber\\
&\gtrsim \int_\Omega(\mu^2+|\D  T_h\tilde v|^2+|\D \tilde v|^2+|\D g|^2)^\frac{p-2}{2}\phi^2 |\D \tilde v_h-\D \tilde v|^2\\
&\quad + \p_z F_\e(\D \tilde v+\D g)\cdot \D (\phi(\tilde v_h-\tilde v))+\p_y \tilde H_\delta(\tilde v+g)\cdot((\phi(\tilde v_h-\tilde v))\d x\\
&\gtrsim \int_{B_\rho(x_0)} |V_{\mu,p}(\D \tilde v_h)-V_{\mu,p}(\D \tilde v)|^2\d x.
\end{align}
To obtain the last line we use Lemma \ref{lem:Vfunctional} and note that by our choice of $h$, we have ${\phi \tilde v_h\in W^{1,q}_0(\Omega)}$.

We continue by estimating
\begin{align*}
|h|^{-1}\int_\Omega \tilde H_\delta(T_h \tilde v+g)-\tilde H_\delta(\tilde v+g)\d x
&\lesssim |h|^{-1}\|H_\delta'\|_{L^\infty(\R)}^2 \int_\Omega |\tilde v_h-\tilde v|\d x\\
&\lesssim \|H_\delta'\|_{L^\infty(\R)}^2\|v\|_{W^{1,p}(\Omega)}\\
&\lesssim \F_\e(v).
\end{align*}
We further estimate, for $h\in C_{\rho_0/4}(\theta_0,\pmb n(x_0))$, that
\begin{align*}
&\quad |h|^{-1}\int_\Omega F_\e(\D T_h \tilde v+\D g)-F_\e(\D \tilde v+\D g)\d x\\
&=|h|^{-1}\int_\Omega F_\e(T_h \D \tilde v +\D \phi(\tilde v_h-v)+\D g)-F_\e(T_h \D \tilde v+\D g)\d x\\
&\qquad+|h|^{-1}\int_\Omega F_\e(T_h \D \tilde v+\D g)-F_\e(\D \tilde v+\D g)\d x\\
=:& A_1 + A_2.
\end{align*}
First using \eqref{ass:growth} and Young's inequality, we find that
\begin{align*}
|A_1|&\lesssim |h|^{-1}\int_\Omega |\D \phi| |\tilde v_h-\tilde v|\big(1+(|T_h \D \tilde v|+|\D \phi\tau_h\tilde v|+|\D g|)^{q-1}\big)\d x\\
&\leq \e_1\int_{B_\sigma(x_0)}  \frac{|\tilde v_h-\tilde v|^q}{|h|^q}\d x+C(\e_1)\int_\Omega (|\D \phi|+|\D \phi|^q)^\frac q {q-1} w_1\d x\\
&\lesssim \e_1\int_{B_{\sigma+\rho_0/4}(x_0)}  |\D \tilde v|^q\d x+C(\e_1)\int_\Omega \left(|\D \phi|+|\D \phi|^\frac{q^2}{q-1}\right) w_1\d x,
\end{align*}
where $w_1=1+|\D g|^q+|\D \tilde v|^q+|\D \tilde v_h|^q+|\tilde v|^q+|\tilde v_h|^q$.

Next we turn to $A_2$. By convexity of $F_\e(\cdot)$, we have
\begin{align}\label{eq:A2}
F_\e(T_h \D \tilde v+\D g)-F_\e(\D \tilde v+\D g)
&\leq (1-\phi)F_\e(\D \tilde v+\D  g)+\phi F_\e(\D \tilde v_h+\D g)-F_\e(\D \tilde v+\D g)\nonumber\\[4pt]
 &= \phi(F_\e(\D \tilde v_h+\D g)-F_\e(\D \tilde v+\D g)).
\end{align}
Thus, we deduce that
\begin{align*}
|A_2|&\leq |h|^{-1}\int_{B_{\sigma}(x_0)} \phi (F_\e(\D \tilde v_h+\D \tilde g_h)-F_\e(\D \tilde v+\D g))\d x\\
&\quad +
|h|^{-1}\int_{B_{\sigma}(x_0)} \phi (F_\e(\D \tilde v_h+\D g)-F_\e(\D \tilde v_h+\D \tilde g_h))\d x\\
&=: B_1 + B_2.
\end{align*}
Here $\tilde g_h$ is a $W^{2,q}(\R^n)$ extension of $g$ to $\R^n$ with $\|\tilde g\|_{W^{2,q}(\R^n)}\lesssim \|g\|_{W^{2,q}(\Omega)}$, the existence of which is guaranteed by \eqref{eq:Gextension}.

Using a change of coordinates and the fact that $\D \tilde v_h=\D \tilde v=0$ in $B_{3\rho_0}\setminus \Omega$, we see that
\begin{align*}
|B_1| &\leq |h|^{-1}\int_{B_{\sigma}(x_0)} \phi (F_\e(\D \tilde v_h+\D \tilde g_h)-F_\e(\D \tilde v+\D \tilde g))\d x\\
&= |h|^{-1}\int_{B_{\sigma}(x_0)+h}\phi(x-h)F_\e(\D \tilde v+ \D \tilde g)\d x-|h|^{-1}\int_{B_{\sigma}(x_0)}\phi(x)F_\e(\D \tilde v+ \D \tilde g)\d x\\
&\leq |h|^{-1}\int_{B_{\sigma+\rho_0/4}(x_0)\setminus B_{3\rho/4}(x_0)}|\phi(x-h)-\phi(x)| F_\e(\D \tilde v+\D \tilde g)\d x\\
&\leq |h|^{-1}\int_{B_{\sigma+\rho_0/4}(x_0)\setminus B_{3\rho/4}(x_0)}|\hat\phi(|x-h|)-\hat\phi(|x|)| F_\e(\D \tilde v + \D \tilde g)\d x\\
\end{align*}
Here $\hat\phi$ is defined through the relation $\phi(x)=\phi'(|x|)$ for $x\in \R^n$.

Using \eqref{ass:growth} and H\"older, we estimate $B_2$ as follows:
\begin{align*}
|B_2|&\leq |h|^{-1}\int_{B_\sigma(x_0))} |\D \tilde g-\D \tilde g_h|(1+|\D\tilde v_h|+|\D \tilde v|+|\D g|+|\D \tilde g_h|)^{q-1}\d x\\[5pt]
&\lesssim |h|^{-1}\|\D g-\D\tilde g_h\|_{L^q(B_{\sigma+\rho_0}(x_0))} \big(1+\|\D \tilde v\|_{L^q(B_\sigma(x_0))}^{q-1}+\|\D\tilde v_h\|_{L^q(B_\sigma(x_0))}^{q-1}\\[5pt]
&\qquad+\|\D \tilde g\|_{L^q(B_\sigma(x_0))}^{q-1}+\|\D \tilde g_h\|_{L^q(B_\sigma(x_0))}^{q-1}\big)\\[5pt]
&\lesssim \|g\|_{W^{2,q}(B_\sigma(x_0))}\big(1+\|\D \tilde v\|_{L^q(B_\sigma(x_0))}^{q-1}+\|\D\tilde v_h\|_{L^q(B_\sigma(x_0))}^{q-1}\\[5pt]
&\qquad+\|\D \tilde g\|_{L^q(B_\sigma(x_0))}^{q-1}+\|\D \tilde g_h\|_{L^q(B_\sigma(x_0))}^{q-1}\big)  = D_1,
\end{align*}
where the last line follows from using the difference quotient characterisation of Sobolev spaces.

Collecting terms, we conclude the estimate
\begin{align}\label{eq:apriori}
&\quad [V_{\mu,p}(\D \tilde v)]_{B^{1/2,2}_\infty(B_\rho)}^2\nonumber\\
&\lesssim\e_1 \int_{B_{\sigma+\rho_0/4}(x_0)}|\D \tilde v|^q\d x+ C(\e_1)\int_{B_{\sigma+\rho_0/4}(x_0)} (|\D \phi|+|\D \phi|^\frac{q^2}{q-1})w_1\d x\nonumber\\
&\quad+|h|^{-1}\int_{B_{\sigma+\rho_0/4}(x_0)\setminus B_{\rho-\rho_0}(x_0)}|\phi(x-h)-\phi(x)| F_\e(\D \tilde v)\d x+\F_\e(v)\nonumber\\
&\lesssim \e_1 \int_{B_{\sigma+\rho_0/4}(x_0)}|\D \tilde v|^q\d x+C(\e_1)\int_{B_{\sigma+\rho_0/4}(x_0)} (|\D \phi|+|\D \phi|^\frac{q^2}{q-1}) w_2\d x\nonumber\\
&\quad+|h|^{-1}\int_{B_{\sigma+\rho_0/4}(x_0)\setminus B_{3\rho/4}(x_0)}|\phi(x-h)-\phi(x)| F_\e(\D \tilde v)\d x + D_1 + \F_\e(v),
\end{align}
where $w_2=1+|V_{\mu,p}(\D \tilde v_\e)|^\frac {2q} p+|V_{\mu,p}(\D \tilde v_{\e,h})|^\frac {2q} p+|V_{\mu,p}(\tilde v_{\e})|^\frac {2q} p + |V_{\mu,p}(\tilde v_{\e,h})|^\frac {2q} p$.

We immediately note that with a choice of $\phi\in C^1_c(B_\sigma)$ such that $|\D \phi|\leq 2/(\sigma-\rho)$, we can conclude that
\begin{align}\label{eq:standardEstimate}
[V_{\mu,p}(\D \tilde v)]_{B^{1/2,2}_\infty(B_\rho)}^2
&\lesssim \frac{1}{\e(\sigma-\rho)^\frac{q^2}{q-1}}\left(1+ \int_{\Omega}F_\e(\D \tilde v)\d x+\|g\|_{W^{2,q}(\Omega)}\right)+\F_\e(v).
\end{align}
Choosing $\rho=\rho_0/2$, $\sigma=\rho$ and applying a standard covering argument, we deduce that $V_{\mu,p}(\D \tilde v)\in B^{1/2,2}_\infty(\R^n)$.
This observation ensures that our calculations below are valid.

Combining \eqref{eq:apriori} and \eqref{eq:basicRadial} with $w = w_2^\frac 1 q$, for any $\e_2>0$, we see that
\begin{align*}
&\quad \int_\Omega 1+|V_{\mu,p}(\D \tilde v)|^2\d x+[V_{\mu,p}(\D \tilde v)]_{B^{1/2,2}_\infty(B_\rho(x_0))}^2\\
&\lesssim
\frac{C(\e_1)}{(\sigma+\rho_0/2-\rho)^{q^2/(q-1)+q/p}}\\
&\quad \times\left(\int_{B_{\sigma+\rho_0/2}(x_0)\setminus B_{\rho-\rho_0/4}(x_0)} 1+|V_{\mu,p}(\D \tilde v)|^2+|V_{\mu,p}(\D \tilde v_h)|^2\d x\right)^{\frac{\xi}{\xi-1}(1-\frac q {\xi p})}\\
&\quad\times\left(\int_{\rho-\rho_0/4}^{\sigma+\rho_0/2} 1+\|V_{\mu,p}(\D \tilde v)\|_{L^{2\xi }(\p B_r(x_0))}^2+\|V_{\mu,p}(\D \tilde v_h)\|_{L^{2\xi}(\p B_r(x_0))}^2\d r\right)^{\frac \xi {\xi-1} \left(\frac q p-1\right)}\\
&\quad\times\left(1+\int_\Omega F_\e(\D \tilde v)\d x\right)+\e_1 \int_\Omega |\D \tilde v|^q\d x+\e_2+D_1+\F_\e(v),
\end{align*}
using the Poincar\'{e} inequality, the bound $\int_\Omega 1+|V_{\mu,p}(\D \tilde v_\e)|^2\d x\geq |\Omega|$ and the fact that $$\frac{\xi}{\xi-1}\left(1-\frac q {\xi p}+\frac q p-1\right) = \frac q p\geq 1.$$

Let $0<\alpha<1/2$ and choose $\rho = \rho_0/2$, $\sigma=\rho_0$. Using Theorem \ref{thm:embedding}, followed by Lemma \ref{lem:FubiniType}, we conclude that
\begin{align}\label{eq:preCovering}
&\quad \int_\Omega 1+|V_{\mu,p}(\D \tilde v)|^2\d x+[V_{\mu,p}(\D \tilde v)]_{B^{1/2,2}_\infty(B_\rho(x_0)\cap\Omega)}^2\nonumber\\
&\lesssim C(\e_1)\rho_0^{-q^2/(q-1)-q/p}\left(\int_{\Omega} 1+|V_{\mu,p}(\D \tilde v)|^2\d x\right)^{\frac{\xi}{\xi-1}(1-\frac q {\xi p})}\nonumber\\
&\quad\times\left(1+\|V_{\mu,p}(\D \tilde v)\|_{L^2(\Omega)}^2+[V_{\mu,p}(\D \tilde v)]_{B^{\alpha,2}_\infty(\Omega)}^2\right)^{\frac \xi {\xi-1} \left(\frac q p-1\right)}\nonumber\\
&\quad \times \left(1+\int_\Omega F_\e(\D \tilde v)\d x\right)+\e_1 \int_\Omega |\D \tilde v|^q\d x+\e_2+D_1+\F_\e(v).
\end{align}
By a standard covering argument and \eqref{eq:besovlocalisation}, we deduce that
\begin{align*}
&\quad \int_\Omega 1+|V_{\mu,p}(\D \tilde v)|^2\d x+[V_{\mu,p}(\D \tilde v)]_{B^{\alpha,2}_\infty(\Omega)}^2\\
&\lesssim C(\e_1)\rho_0^{-q^2/(q-1)-q/p}\left(\int_{\Omega} 1+|V_{\mu,p}(\D \tilde v)|^2\d x\right)^{\frac{\xi}{\xi-1}(1-\frac q {\xi p})}\\
&\quad\times\left(1+\|V_{\mu,p}(\D \tilde v_\e)\|_{L^2(\Omega)}^2+[V_{\mu,p}(\D \tilde v)]_{B^{\alpha,2}_\infty(\Omega)}^2\right)^{\frac \xi {\xi-1} \left(\frac q p-1\right)}\\
&\quad \times \left(1+\int_\Omega F_\e(\D \tilde v)\d x\right)+\e_1 \int_\Omega |\D \tilde v|^q\d x+\e_2\\
&\quad + \|g\|_{W^{2,q}(\Omega)}\left(1+\|\D v\|_{L^q(\Omega)}^{q-1}+\|\D g\|_{L^q(\Omega)}^{q-1}\right)+\F_\e(v).
\end{align*}
Next fix $0<\beta<1$.
We note that, by interpolation between $L^p(\Omega)$ and $L^\frac{np}{n-\beta}(\Omega)$,
\begin{align*}
\|\D v\|_{L^q(\Omega)}^{q-1}&\leq \|\D v\|_{L^p(\Omega)}^{(1-\theta)(q-1)}\|\D v\|_{L^\frac{np}{n-\beta}(\Omega)}^{\theta(q-1)}\\
&\leq \|V_{\mu,p}(\D v)\|_{L^2(\Omega)}^{2(1-\theta)(q-1)/p}\|V_{\mu,p}(\D v)\|_{L^\frac{2n}{(n-\beta)}(\Omega)}^{2\theta(q-1)/p}
\end{align*}
with $\theta = \frac{n}{\beta}\left(1-\frac p q\right)$. Since $q<p+1$, it is straightforward to check that $\frac{2\theta(q-1)} p<2$ for $\beta$ chosen sufficiently close to $1$.

Since $\frac q p< 1+ \frac 1 {n-1}$, using the definition of $\xi$ we find that
\begin{align*}
\frac{\xi}{\xi-1}\frac{q-p}{p} <1.
\end{align*}
Choosing $\e_1$ sufficiently small and $\e_2 =1$, after re-arranging and using the embedding of $B^{\alpha,2}_\infty(\Omega)$ into $W^{1,2n/(n-\beta)}(\Omega)$ (which holds for a sufficiently large choice of $\alpha$), we deduce the estimate
\begin{align}\label{eq:autonBesov}
[V_{\mu,p}(\D \tilde v)]_{B^{\alpha,2}_\infty(\Omega)}^2\lesssim \left(1+\|V_{\mu,p}(\D \tilde v)\|_{L^2(\Omega)}^2+\|g\|_{W^{2,q}(\Omega)}^q\right)^\frac{\tilde \beta} {1-\tilde\beta}\left(1+\int_\Omega F_\e(\D \tilde v)\d x\right)^\frac{1}{1-\tilde\beta}
\end{align}
where $\tilde\beta = \max\left(\dfrac{\xi}{\xi-1}\dfrac{q-p}{p},\dfrac{2n(q-1)(q-p)}{\beta pq}\right)$.

As $\|V_{\mu,p}(\D \tilde v)\|_{L^2(\Omega)}^2\lesssim 1+ \F_{\e,\delta}(\D \tilde v)+\|g\|_{W^{2,q}(\Omega)}$, using Lemma \ref{lem:lavrentiev} we conclude that the bound is independent of $\e$.

Finally, by Sobolev embedding, for any $\beta<1$, we can find $\alpha\in(0,1/2)$ such that $\|\D v\|_{L^\frac{np}{n-\beta}(\Omega)}\lesssim \|\D v\|_{B^{2\alpha/p,p}_\infty(\Omega)}\lesssim \|V_{\mu,p}(\D  v)\|_{B^{\alpha,2}_\infty(\Omega)}$. Considering the regularity of $g$, the result now follows.
\end{proof}

\begin{remark}
We have chosen to present the argument leading to a global apriori estimate. However, it is straightforward to adapt our argument in order to obtain a local version of the estimate. This is achieved by considering \eqref{eq:preCovering}. Instead of a covering argument, the estimates follow arguing along similar lines as before but requires an iteration lemma. See, for example, \cite{DeFilippis2021a} for similar arguments.
\end{remark}

\remark{Our results also imply an improved differentiability result. For this we return to \eqref{eq:besovestimate2}. Due to the $W^{1,q}$-regularity of $u_{\e,\delta}$, the right-hand side remains bounded in the limit as $\e,\delta\to 0$. Thus we obtain $B^{1+\frac \alpha p}_\infty(\Omega)$-regularity of $u_{\e,\delta}$.}

\subsection{The non-autonomous case}
For the non-autonomouse case we prove the following result.
\begin{proposition}\label{prop:nonauton}
Let $\alpha\in(0,1]$ and
$2\leq p\leq q<\frac{(n+\alpha)p}{n}$.
Suppose that $g\in W^{1+\alpha,q}(\Omega)$ and $\psi\in W^{2,\infty}(\Omega)$. Let $F\equiv F(x,z)$ satisfy \eqref{ass:elliptic},\eqref{ass:growth} and \eqref{def:bounds3}. Then, for minimisers $u_{\e,\delta}$ of $\F_{\e,\delta}(\cdot)$, for any $\tau\in(0,\alpha)$ and some $\beta>0$, we have that
\begin{align*}
\|\D u_{\e,\delta}\|_{L^q(\Omega)}^q&\lesssim[V_{\mu,p}(\D u_{\e,\delta})]_{B^{\tau,2}_\infty(\Omega)}^2\\
&\lesssim \left(1+\|V_{\mu,p}(\D u_{\e,\delta})\|_{L^2(\Omega)}^2+\|g\|_{W^{2,q}(\Omega)}^q\right)^\beta,
\end{align*}
where the implicit constant is independent of $\e,\delta$.
\end{proposition}

\begin{proof}
We employ the notation of the proof of Proposition \ref{prop:auton}.
The proof follows the same outline as the autonomous case. However the trick of applying the Sobolev embedding on spheres cannot be applied and hence the proof is simpler. We state only the key steps.

We argue exactly as before up to \eqref{eq:A2}.
Here, using the convexity of $F_\e(\cdot)$ as before, we find that
\begin{align*}
|A_2|
 &\leq \int_{B_{3\rho_0}(x_0)}\tau_{-h}\phi F_\e(x-h,\D \tilde v+\D \tilde g_{-h})\\
& \quad+\phi(x)\big(F_\e(x-h,\D(\tilde g_{-h}+\tilde v))-F_\e(x-h,D(\tilde g-\tilde v))\big)\d x\\
&\quad+\int_{B_{2\rho_0}(x_0)}\phi(x)\big(F_\e(x-h,\D \tilde v+\D \tilde g)-F_\e(x,\D \tilde v+\D \tilde g)\big)\d x.
\end{align*}
Using \eqref{ass:growth} and \eqref{def:bounds3}, the regularity of $\phi$ and $g$ and \eqref{eq:Gextension}, we estimate each term in turn to conclude that
\begin{align*}
|A_2|&\lesssim |h|^\alpha\left(1+\|\D v\|_{L^q(\Omega)}^q+\|g\|_{W^{1+\alpha,q}(\Omega)}^q\right).
\end{align*}
Combining this with the estimates as obtained in the autonomous case, we find that, $x_0\in \R^n$,
\begin{align*}
&\quad \sup_{h\in C_{\rho_0}(\theta_0,\pmb n(x_0))}|h|^{-\alpha}\|V_{\mu,p}(\D v)-V_{\mu,p}(\D  T_h\tilde v_h)\|_{L^2(B_{\rho_0}(x_0))}^2\\
&\lesssim 1+\| \D v\|_{L^q(\Omega)}^q+\|g\|_{W^{1+\alpha,q}(\Omega)}^q.
\end{align*}
Recalling the definition of $T_h$, using the triangle inequality and regularity of $\tilde g$, we see that
\begin{align}\label{eq:basicBesov}
&\quad \sup_{h\in C_{\rho_0}(\theta_0,\pmb n(x_0))}|h|^{-\alpha}\|\D\tilde v_h-D\tilde v\|_{L^p(B_{\rho_0}(x_0))}^p\\
&\lesssim 1+\| \D v\|_{L^q(\Omega)}^q+\|g\|_{W^{1+\alpha,q}(\Omega)}^q.
\end{align}
Using the characterisation of Besov spaces (\ref{eq:besovcharacterisation}), we conclude that for every $x_0\in \R^n$,
\begin{align}
[\D v]_{\frac \alpha p,p,B_{\rho_0}(x_0)}^p\lesssim \left(1+\|\D v\|_{L^q(\Omega)}^q+\|g\|_{W^{1+\alpha,q}(\Omega)}^q\right).
\end{align}
Covering $\Omega$ by a finite number of balls of radius $\rho_0$, using (\ref{eq:besovlocalisation}) we conclude that
\begin{align}\label{eq:besovestimate2}
\|v\|_{B^{1+\frac \alpha p,p}_\infty(\Omega)}^p\lesssim \left(1+\|v\|_{W^{1,q}(\Omega)}^q+\|g\|_{W^{1+\alpha,q}(\Omega)}^q\right).
\end{align}
We recall that $B^{1+\frac \alpha p,p}_\infty(\Omega)$ embeds continuously into $W^{1,\frac{np}{n-\beta}}(\Omega)$ for any $\beta <\alpha$ by Theorem \ref{thm:embedding}. Hence we choose $\beta$ such that $q< \frac{p(\beta+n)}{n}$ and use \eqref{eq:interpolationInequality} with $\theta = \frac{np}{\beta}\left(\frac 1 p-\frac 1 q\right)$ to see that
\begin{align}\label{eq:interpolation}
\|Dv\|_{L^q(\Omega)}\leq \|Dv\|_{L^p(\Omega)}^{1-\theta} \|Dv\|_{L^{\frac{np}{n-\beta}}(\Omega)}^{\theta}.
\end{align}
As $q<\frac{(n+\beta)p}{n}$, it follows that $q\theta<p$. Using \eqref{eq:interpolation} in \eqref{eq:besovestimate2}, we find after using Young's inequality that
\begin{align*}
\|v\|_{W^{1,\frac{np}{n-\beta}}(\Omega)}^p\lesssim 1+\frac 1 2\|v\|_{W^{1,\frac{np}{n-\beta}}(\Omega)}^p+C(\theta)\|v\|_{W^{1,p}(\Omega)}^\frac{\theta q}{(\theta q-p)}+\|g\|_{W^{1+\alpha,q}(\Omega)}^q.
\end{align*}
Rearranging, as well as recalling \eqref{def:lower} and the regularity of $g$, we obtain the desired result.
\end{proof}

\remark{The results of this section can easily be extended to the case $1<p\leq 2$. We refer to \cite{Koch2020} for the details.}

\remark{Our results also imply an improved differentiability result. For this we return to \eqref{eq:besovestimate2} and take limits as $\e\to 0$. Due to the $W^{1,q}$-regularity of $u_{\e,\delta}$, the right-hand side remains bounded. Thus we obtain $B^{1+\frac \alpha p}_\infty(\Omega)$-regularity of $u_{\e,\delta}$.}

\section{Proof of Theorem \ref{thm:main}}\label{sec:mainproof}
We are now in a position to prove Theorem \ref{thm:main} and Proposition \ref{prop:exact}.
\begin{proof}[Proof of Theorem \ref{thm:main}]
For minimisers $u_{\e,\delta}$ of $\F_{\e,\delta}(\cdot)$, due to minimiality as well as Proposition \ref{prop:auton} and Proposition \ref{prop:nonauton}, we have the bound
\begin{align}\label{eq:bound}
\|u_{\e,\delta}\|_{W^{1,q}(\Omega)}\leq C<\infty
\end{align}
where $C$ is independent of $\e$, $\delta$. Next, we recall that, as $\delta\to 0$, we have $u_{\e,\delta}\to u_\e$ in $W^{1,q}(\Omega)$ where $u_\e$ is a minimiser of
\begin{align*}
\int_\Omega F(x,\D v)+\e |\D v|^q+\kappa(\psi-v)_+\d x
\end{align*}
in the class $W^{1,q}_g(\Omega)$. Using Proposition \ref{prop:exact}, we see that if $\kappa$ is chosen sufficiently large, independently of $\e$, $u_\e$ minimises
\begin{align*}
\int_\Omega F(x,\D v)+\e |\D v|^q\d x
\end{align*}
in the class $K^{*,\psi}_g(\Omega)$. Due to Lemma \ref{lem:lavrentiev}, $u_\e \to u$ in $W^{1,p}(\Omega)$, where $u$ is the relaxed minimiser of \eqref{obstacleProblem}. In particular, we can choose a suitable diagonal subsequence of $u_{\e,\delta}$, denoted $u_k$, with $u_k\to u$ as $k\to \infty$. Passing to the limit in \eqref{eq:bound} we deduce that $u\in W^{1,q}(\Omega)$. 
\end{proof}

\begin{proof}[Proof of Proposition \ref{prop:auton}]
Due to Lemma \ref{lem:nonLavrentiev} and Lemma \ref{lem:nonlavrentievNonAuton}, given a pointwise minimiser $u\in K^\psi_g(\Omega)$ of \eqref{obstacleProblem}, we can find $(u_j)\subset K^{*,\psi}_g(\Omega)$ and  $F^j\equiv F^j(x,z)\colon \Omega\times \R^{N\times n}\to \R$ satisfying the same assumptions as $F$, with constants in the various bounds that do not depend on $j$ such that
\begin{align}\label{eq:bound2}
u_j \to u \text{ in } W^{1,p}(\Omega) \quad\text{ and }\quad \int_\Omega F^j(x,\D u_j)\d x\to \F(u) \text{ as } j\to\infty.
\end{align} 

Since $F^j(x,z)$ satisfies the same assumptions as $F$, we can apply Proposition \ref{prop:auton} and Proposition \ref{prop:nonauton} to see that minimisers $u_{\e,\delta,j}$ of $$\F^j_{\e,\delta}:=\int_\Omega F^s(x,\D u)+\e|\D u|^q+\kappa H_\delta(\psi-u)\d x$$ satisfy
\begin{align}\label{eq:bound3}
\|u_{\e,\delta,j}\|_{W^{1,q}(\Omega)}\leq C<\infty
\end{align}
independent of $\e,\delta,j$. Thus, repeating the arguments of the proof of Theorem \ref{thm:main} we extract a subsequence $v_k=u_{\e_k,\delta_k,j_k}$ with $\e_k,\,\delta_k\to 0$ and $j_k\to\infty$ as $k\to\infty$ that converges weakly in $W^{1,q}(\Omega)$ to the minimiser $v$ and such that
\begin{align*}
\F^{j_k}_{\e_k,\delta_k}(v_k^j)\to \overline\F(v)
\end{align*}
as $k\to\infty$. Moreover, by weak lower semicontinuity of norms and \eqref{eq:bound3}, we have ${v\in W^{1,q}(\Omega)}$ and so $\overline \F(v)=\F(v)$. By weak semicontinuity and minimality of $u_{\e,\delta,j}$, we note that
\begin{align*}
\F(v) &= \lim_{k\to\infty} \F^{j_k}_{\e_k,\delta_k}(v_k)
\leq \lim_{k\to\infty} \F^{j_k}_{\e_k,\delta_k}(u_{j_k})
= \F(u).
\end{align*}
For a sufficiently  large choice of $\kappa$ we have $v\geq \psi$ by Proposition \ref{prop:exact}. Using the minimality of $u$ in the class $K^\psi_g(\Omega)$, we conclude that $\F(v)=\F(u)$. By convexity of $\F(\cdot)$, it follows that $v=u$, concluding the proof.
\end{proof}

\bibliographystyle{plain}
\bibliography{../Refs/bibtex/pqboundary}
\end{document}